\documentclass{amsart}
\usepackage{amscd,amssymb,euscript}
\usepackage[all]{xy}

\newcommand\A{\mathbb A}
\renewcommand\P{\mathbb P}
\renewcommand\phi{\varphi}

\DeclareMathOperator\Bl{{Bl}}
\DeclareMathOperator\Gr{{Gr}}

\DeclareMathOperator\Isom{{Isom}}
\DeclareMathOperator\Coker{{Coker}}
\DeclareMathOperator\spec{{Spec}}

\DeclareMathOperator{\Id}{Id}
\DeclareMathOperator{\Mor}{Mor}
\DeclareMathOperator{\HOM}{\mathcal{H}om}

\newcounter{noindnum}[subsection]
\renewcommand{\thenoindnum}{\roman{noindnum}}
\newcommand{\noindstep}{\refstepcounter{noindnum}{\rm(}\thenoindnum\/{\rm)} }
\newcommand{\stepzero}{\setcounter{noindnum}{0}
}
\theoremstyle{plain}
\newtheorem{theorem}{Theorem}
\newtheorem{proposition}[subsection]{Proposition}
\newtheorem{lemma}[subsection]{Lemma}

\theoremstyle{definition}
\newtheorem{definition}[subsection]{Definition}

\theoremstyle{remark}
\newtheorem{remark}[subsection]{Remark}
\newtheorem*{remark*}{Remark}

\newtheorem*{remarks*}{Remarks}
\newtheorem{convention}[subsection]{Convention}

\newcommand\fm{{\mathfrak{m}}}

\newcommand\fp{{\mathfrak{p}}}
\newcommand\fq{{\mathfrak{q}}}

\newcommand\cB{{\mathcal{B}}}

\newcommand\cE{{\mathcal{E}}}
\newcommand\cF{{\mathcal{F}}}
\newcommand\cG{{\mathcal{G}}}
\newcommand\cH{{\mathcal{H}}}
\newcommand\cI{{\mathcal{I}}}
\newcommand\cL{{\mathcal{L}}}

\newcommand\cO{{\mathcal{O}}}

\newcommand\cT{{\mathcal{T}}}
\newcommand\cV{{\mathcal{V}}}
\newcommand\cW{{\mathcal{W}}}
\newcommand\cS{{\mathcal{S}}}
\newcommand\cY{{\mathcal{Y}}}
\newcommand\cX{{\mathcal{X}}}
\newcommand\cZ{{\mathcal{Z}}}

\newcommand\bB{{\mathbf{B}}}
\newcommand\bG{{\mathbf{G}}}
\newcommand\bH{{\mathbf{H}}}
\newcommand\bO{{\mathbf{O}}}
\newcommand\bSO{{\mathbf{SO}}}
\newcommand\bT{{\mathbf{T}}}
\newcommand\bU{{\mathbf{U}}}

\newcommand\Z{\mathbb Z}

\title[On the Grothendieck--Serre conjecture in mixed characteristic]{On the Grothendieck--Serre conjecture on principal bundles in mixed characteristic}
\author{Roman Fedorov}
\email{rmfedorov@gmail.com}
\address{Kansas State University, Manhattan, KS, USA}
\address{Max Planck Institute for Mathematics, Bonn, Germany}
\address{University of Pittsburgh, Pittsburgh, PA}

\begin{document}

\begin{abstract}
Let $R$ be a regular local ring. Let $\bG$ be a reductive $R$-group scheme. A conjecture of Grothendieck and Serre predicts that a principal $\bG$-bundle over~$R$ is trivial if it is trivial over the quotient field of $R$. The conjecture is known when~$R$ contains a field. We prove the conjecture for a large class of regular local rings \emph{not} containing fields in the case when $\bG$ is split.
\end{abstract}

\vspace{3cm}

\maketitle

\section{Introduction and main results}
Let $R$ be a regular local ring; let $\bG$ be a reductive group scheme over $R$. A~ conjecture of Grothendieck and Serre
(see~\cite[Remarque, p.31]{SerreFibres}, \cite[Remarque~3, p.26-27]{GrothendieckTorsion}, and~\cite[Remarque~1.11.a]{GrothendieckBrauer2}) predicts that a principal $\bG$-bundle over $R$ is trivial, if it is trivial over the fraction field of $R$. Recently this has been proved in the case when $R$ contains a field in~\cite{FedorovPanin}, it was extended to the case of finite fields in~\cite{PaninFiniteFieldsIzvestiya}. In this paper we consider the case when $R$ \emph{contains no field}, that is, the case of \emph{mixed characteristic}.

Note that a regular local ring $R$ contains no field if and only if there is a prime number~$p$ (necessarily unique) such that~$p$ is neither invertible nor zero in $R$. In this case $R$ contains the localization $\Z_{(p)}$ of $\Z$ at the prime ideal~$(p)=p\Z$.

Thus, we assume that $R$ is a $\Z_{(p)}$-algebra. \emph{We will also assume that $R/pR$ is a regular ring.} In this case a theorem of Popescu~\cite{Popescu,SpivakovskyPopescu,SwanOnPopescu} reduces the question to the case when $R$ is a localization of a finitely generated smooth $\Z_{(p)}$-algebra $A$ at a maximal ideal. Taking the closure of $\spec A$ in $\P_{\Z_{(p)}}^N$, we may assume that $R$ is the local ring of a closed point $x$ on an integral scheme $X$ projective over $\Z_{(p)}$.

Additionally, we will assume that (I) \emph{the fiber $X_p$ is generically reduced, and that \emph{(II)} the set of points where $X$ is not regular intersects $X_p$ in a subset of codimension at least two in $X_p$}. Note that condition~(I) is satisfied if the fiber $X_p$ is irreducible because the projection is smooth at $x$. On the other hand, both conditions are satisfied if the set of points where the projection $X\to\spec\Z_{(p)}$ fails to be smooth, has codimension at least 3 in $X$.

Below we will prove the conjecture of Grothendieck and Serre under the above assumptions when the group scheme $\bG$ is split; see Theorem~\ref{th:GrSerre}. We work in a slightly greater generality: we weaken condition~(I) and we consider projective schemes over any excellent discrete valuation ring $\Lambda$, not just $\Z_{(p)}$-schemes. In particular, $\Lambda$ can be a localization of any number ring at a maximal ideal.

We note that previously the conjecture was known in a very few mixed characteristic cases, namely, when $\bG$ is a torus~\cite{ColliotTheleneSansuc}, when $\dim R=1$, when $R$ is Henselian~\cite{Nisnevich1}. Next, the case of $\bG=\mathbf{PGL}(n,R)$ follows from a similar statement for Brauer groups~\cite[Thm.~4.3]{ColliotTheleneSansuc} (more generally, one can derive the statement for $\bG=\mathbf{PGL}(A)$, where $A$ is an Azumaya algebra over $R$). Also, in~\cite{Nisnevich2} the conjecture
is proved when $\bG$ is quasisplit and $\dim R=2$ but there it is assumed that the residue field of $R$ is infinite. Thus our results are new even in dimension two. We also note that while the current paper was under review, the subject was further developed, see~\cite{CesnaviciusGrSerre}.

\subsection{Definitions and conventions}
A group scheme $\bG$ over a scheme $S$ is called \emph{reductive\/} if it is affine and smooth as an $S$-scheme and if, moreover, all its geometric fibers are connected reductive algebraic groups. This definition of a reductive $R$-group scheme coincides with~\cite[Exp.~XIX, Def.~2.7]{SGA3-3}.

A reductive group scheme $\bG$ over a local scheme $S$ is \emph{split\/} if it contains a maximal torus $\bT\subset\bG$ such that $\bT\simeq({\mathbb G}_{m,S})^r$ for some $r$ (cf.~\cite[Exp.~XXII, Prop.~2.2]{SGA3-3}). Note that such a group scheme comes as a pullback from $\spec\Z$ (see~\cite[Exp.~XXV, Thm.~1.1]{SGA3-3}).

Let $\bG$ be a group scheme faithfully flat and finitely presented over $S$. An $S$-scheme~$\cG$ with a left action of $\bG$ is \emph{a principal $\bG$-bundle over $S$}, if $\cG$ is faithfully flat and finitely presented over $S$, and the morphism $\bG\times_S\cG\to\cG\times_S\cG$, whose first component is the action and the second is the projection, is an isomorphism (see~\cite[Sect.~6]{FGA1}). A principal $\bG$-bundle $\cE$ over $S$ is \emph{trivial\/} if there is an isomorphism of $S$-schemes $\cE\simeq\bG$ compatible with the action of $\bG$, where $\bG$ acts on itself by left multiplication. A principal $\bG$-bundle is trivial if and only if it has a section as an $S$-scheme.

If $T$ is an $S$-scheme, we will use the term ``principal $\bG$-bundle over $T$'' to mean a principal $\bG\times_ST$-bundle over $T$. We usually skip the adjective `principal' as we are only considering principal $\bG$-bundles.

Assume that $\bG$ is affine over $S$. In this case, we denote by $H^1_{\text{fppf}}(S,\bG)$ the pointed set of isomorphism classes of $\bG$-bundles over $S$  (as every such bundle is locally trivial in the fppf topology). The subset corresponding to \'etale locally trivial bundles is denoted by $H^1_{\text{\'et}}(S,\bG)$. We note that if $\bG$ is smooth over $S$, then we have
\[
    H^1_{\text{\'et}}(S,\bG)=H^1_{\text{fppf}}(S,\bG).
\]
If $T$ is an $S$-scheme and $s\in S$ is a point, we write $T_s$ for the fiber $T\times_Ss$. We write $k(s)$ for the residue field of $s$.

The symbol `$\simeq$' means that two objects are isomorphic; we use the equality `$=$' to emphasize that the isomorphism is canonical. We use boldface font for group schemes (e.g~$\bG$, $\bB$, etc.) and the calligraphic font for principal bundles (e.g.~$\cG$, $\cE$, etc.).

The notation $\#A$ stands for the number of elements of the finite set $A$.

\subsection{Main result}\label{sect:MainRes}
\begin{theorem}\label{th:GrSerre}
Let $\Lambda$ be an excellent discrete valuation ring; let $b\in\spec\Lambda$ be the closed point. Let $X$ be an integral scheme and $\pi\colon X\to\spec\Lambda$ be a flat projective morphism. Denote by $X^{sing}$ the set of $y\in X$ such that the local ring $\cO_{X,y}$ is not a regular ring. Assume that $\pi\colon X\to\spec\Lambda$ satisfies the following properties

\hspace*{1cm} (I) The smooth locus of $X_b$ is dense in $X_b$.

\hspace*{1cm} (II) The intersection $X^{sing}\cap X_b$ has codimension at least two in $X_b$.

Let $x\in X$ be a closed point such that $\pi$ is smooth at $x$. Let $\bG_{X,x}$ be a split reductive $\cO_{X,x}$-group scheme. Then a principal $\bG_{X,x}$-bundle over $\cO_{X,x}$ is trivial, if it has a rational section.
\end{theorem}

The proof of the theorem occupies Sections~\ref{sect:ProofA1}--\ref{sect:EndOfProof}.

\begin{remarks*}

$\bullet$ The set $X^{sing}$ is closed in $X$, since $\Lambda$ is excellent; see~\cite[Scholie~7.8.3(iv)]{EGAIV-2}.

$\bullet$ We note that $X^{sing}\cap X_b$ is in general smaller than the set of points where $X_b$ is not regular.

$\bullet$ The condition that $\Lambda$ is excellent is not needed. Indeed, the Grothendieck--Serre conjecture is known for regular local rings containing finite fields~\cite{PaninFiniteFieldsIzvestiya,PaninNiceTriples,PaninPurity17}. Thus we may assume that $\Lambda$ does not contain a finite field. In this case $\Lambda$ is automatically excellent; see~\cite[Scholie~7.8.3(iii)]{EGAIV-2}. However, we prefer to keep this assumption in order to have our theorem independent from Panin's results~\cite{PaninFiniteFieldsIzvestiya,PaninNiceTriples,PaninPurity17}.

$\bullet$ Condition (I) is satisfied if $X_b$ is irreducible, because $\pi$ is smooth at $x$.

$\bullet$ If the residue field of $b$ is perfect, then Condition (I) is equivalent to the condition that $X_b$ has no multiple components.

$\bullet$ We expect that, more generally, the theorem and its proof hold for the semi-local rings of finitely many closed points on $X$.
Note that the conjecture is proved in the case of semi-local Dedekind domains in~\cite{Guo2019GrSerreDedekind}, which extends the results of~\cite{Nisnevich1}. See also~\cite{PaninStavrovaDedekind} in the split case.
\end{remarks*}

The following result of independent interest will be used in the proof.
\begin{theorem}\label{th:A1}
Let $R$ be a Noetherian local ring. Let $\bH$ be a split reductive group scheme over $R$.
Let $\cF$ be a principal $\bH$-bundle over $\A^1_R:=\spec R[t]$ such that $\cF$ is trivial over the complement of a closed subscheme that is finite over $\spec R$. Then $\cF$ is trivial.
\end{theorem}

This theorem is similar to~\cite[Thm.~1.3]{PaninStavrovaVavilov} and to~\cite[Thm.~3]{FedorovPanin}. It will be proved in Section~\ref{sect:ProofA1}. Note that the ring $R$ is not required to be regular.

\subsection{Example: quadratic forms} We have the following relative result. Let $R$ be a regular local ring and let the $R$-group scheme $\bSO_n$ be the split form of the special orthogonal group scheme.

\begin{theorem}\label{th:QuadForms}
Let $R$ be a regular local ring such that the Grothendieck--Serre conjecture holds for $R$ and $\bSO_{2n}$. Assume that 2 is invertible in $R$. Let $Q=\sum_{i,j}q_{ij}x_ix_j$ and $Q'=\sum_{i,j}q'_{ij}x_ix_j$ be quadratic forms in $n$ variables with coefficients in $R$ such that their discriminants are invertible in $R$. Assume that there is a linear transformation with coefficients in the fraction field of $R$, taking $Q$ to $Q'$. Then there is a linear transformation with coefficients in $R$ taking $Q$ to $Q'$.
\end{theorem}

When 2 is not necessarily invertible in $R$ we have the following result. Define the split quadratic form over $R$ as follows
\[
    Q_n=x_1x_{m+1}+\dotsb+x_mx_{2m}\text{ if }n=2m
\]
and
\[
    Q_n=x_1x_{m+1}+\dotsb+x_mx_{2m}+x_{2m+1}^2\text{ if }n=2m+1.
\]

Note that $\bSO_n$ is the special orthogonal group scheme associated to $Q_n$ (see~\cite[Ch.~IV, Sect.~5]{KnusBook} for the correct definition in the case when 2 is not invertible in $R$). Recall (see e.g.~\cite[Ch.~IV, Sect.~3]{KnusBook}) that if $n$ is odd and $Q$ is a quadratic form with coefficients in $R$, then one can define its half-discriminant (which is just $1/2$ times the discriminant if 2 is invertible in $R$).

\begin{theorem}\label{th:QuadForms2}
Let $R$ be a regular local ring such that the Grothendieck--Serre conjecture holds for $R$ and $\bSO_n$. Let $Q=\sum_{i,j}q_{ij}x_ix_j$ be a quadratic form in $n$ variables with coefficients in $R$ such that its discriminant is invertible in $R$ if $n$ is even, and its half-discriminant is invertible in $R$ if $n$ is odd. Assume that there is a linear transformation with coefficients in the fraction field of $R$, taking $Q$ to $Q_n$. Then there is a linear transformation with coefficients in $R$ taking $Q$ to $Q_n$.
\end{theorem}

Note that, if $X$ and $x$ are as in Theorem~\ref{th:GrSerre}, then the conditions of the above two theorems are satisfied for $R=\cO_{X,x}$.
Theorems~\ref{th:QuadForms} and~\ref{th:QuadForms2} are proved in the last section.

\subsection{Outline of the paper} We start by proving Theorem~\ref{th:A1} in Section~\ref{sect:ProofA1}. After that we proceed with the proof of Theorem~\ref{th:GrSerre}. Let us give a brief overview of the proof. By~\cite{Nisnevich1} we may assume that the relative dimension of the flat morphism $X\to\spec\Lambda$ is at least one.

The fist step in the proof is to fiber a neighborhood of $x$ in $X$ into curves. Thus we choose an appropriate neighborhood $X'$ of $x$ in $X$ and a smooth fibration $X'\to S$ of relative dimension one, having some nice properties (see Definition~\ref{def:QEF} below). We extend $\cG$ to a principal bundle $\cF$ over $X'$ such that $\cF$ is trivial over the complement of a subscheme finite over $S$. This step, carried out in Section~\ref{sect:QElFib} differs crucially from the equal characteristic case. In particular, we use the fact that a generically trivial principal bundle can be reduced to a Borel subgroup on the complement of a codimension two subscheme, see Proposition~\ref{pr:codim2}.

In Section~\ref{sect:EndOfProof} we complete the proof of Theorem~\ref{th:GrSerre} as follows. We pull $\cF$ back to an open subset of $X'\times_S U$, where $U:=\spec\cO_{X,x}$. Then, we descend the bundle obtained to $\A^1_U$, employing the theory of nice triples of Panin (cf.~\cite[Def.~3.1]{PaninStavrovaVavilov} and Definition~\ref{def:nice} below), reducing Theorem~\ref{th:GrSerre} to Theorem~\ref{th:A1}. See Remark~\ref{rm:nicetriples} regarding the rationale for using nice triples.

In Section~\ref{sect:QuadForms} we prove Theorems~\ref{th:QuadForms} and~\ref{th:QuadForms2}.

\subsection{Acknowledgments} The author is very grateful to Ivan Panin for introducing him to the subject and to the techniques. In fact, the main ideas of this paper grew out of a conversation with Panin at the University of Mainz. The author also wants to thank Dima Arinkin and Kestutis \v{C}esnavi\v{c}ius for useful comments. The author is particularly grateful to Brian Conrad for explaining to him how to obtain a result about quadratic forms in the case when 2 is not invertible in the ring (Theorem~\ref{th:QuadForms2}) and for other comments. The author thanks the anonymous referees for very valuable suggestions. One of the referees explained to the author how to simplify certain proofs and how to strengthen the statement of Theorem~\ref{th:QuadForms}.

The author is partially supported by NSF grants DMS-1406532 and DMS-2001516. A major part of the paper was written, while the author held a fellowship at Max Planck Institute for Mathematics in Bonn. He wants to thank the Institute for the hospitality.

\section{Bundles over $\A^1$: Proof of Theorem~\ref{th:A1}}\label{sect:ProofA1}

\subsection{Horrocks type statement}
The following statement and its proof are close to~\cite[Thm.~9.6]{PaninStavrovaVavilov}.

\begin{proposition}\label{pr:Horrocks}
Let $R$ be a Noetherian local ring, $U:=\spec R$. Let $x\in U$ be the closed point. Let~$\bH$ be a reductive group scheme over $U$ such that there is an embedding $\bH\to\mathbf{GL}(n,U)$ for some positive integer $n$. Let $\P^1_x$ be the $x$-fiber of the projection $\P^1_U\to U$, let $\bH_x$ be the $x$-fiber of $\bH$. Let $\cH$ be an $\bH$-bundle over $\P^1_U$ such that its restriction to $\P^1_x$ is a~trivial $\bH_x$-bundle. Then $\cH$ is isomorphic to the pullback of an $\bH$-bundle over $U$.
\end{proposition}
\begin{proof}
By~\cite[Cor.~6.12(ii)]{ColliotTS1979fibres} the quotient $\cX:=\mathbf{GL}(n,U)/\bH$ is represented by an affine scheme.

Consider the associated $\mathbf{GL}(n,U)$-bundle $\cH':=\mathbf{GL}(n,U)\times^\bH\cH$. Let, under the equivalence between $\mathbf{GL}(n,U)$-bundles and rank $n$ locally free coherent sheaves, $\cH'$ correspond to the sheaf $\cS$ on $\P_U^1$. By the assumption on $\cH$, $\cS_x$ is isomorphic to $\oplus_{i=1}^n\cO_{\P_x^1}$. Thus
\[
    H^1(\P^1_x,\HOM(\cS_x,\cS_x))=H^1(\P^1_x,\oplus_{i=1}^{n^2}\cO_{\P^1_x})=0.
\]
Therefore, according to~\cite[Cor.~4.6.4]{EGAIII-1}, $\cS$ is a free sheaf. Thus $\cH'$ is trivial.

Consider the morphism of exact sequences, induced by the canonical projection $pr_U\colon\P^1_U\to U$,
\[
\begin{CD}
\Mor_U(U,\cX) @>\partial>> H^1_{\text{\'et}}(U,\bH) @>>> H^1_{\text{\'et}}(U,\mathbf{GL}(n,U))\\
@V pr_U^* VV @VVV @VVV\\
\Mor_U(\P^1_U,\cX) @>\partial>> H^1_{\text{\'et}}(\P^1_U,\bH) @>>> H^1_{\text{\'et}}(\P^1_U,\mathbf{GL}(n,U)).
\end{CD}
\]
The class of $[\cH]\in H_{\text{\'et}}^1(\P^1_U,\bH)$ is in the image of $\partial$, because $\cH'$ is trivial. It remains to show that the morphism $pr_U^*$ is surjective. This follows from~\cite[Prop.~6.1]{MumfordEtAl1994GIT}. This proposition is applicable because $\P^1_x$ is projective and $\cX_x$ is affine, so any morphism $\P^1_x\to\cX_x$ must be constant.
\end{proof}

\subsection{Gluing principal bundles}\label{sect:glue}
As before, let $R$ be a Noetherian local ring, $U:=\spec R$. Let~$\bH$ be a split reductive group scheme over $U$.
Let $Y=0\times U$ be the zero section in $\P_U^1$. Let $D_Y:=\spec R[[t]]$ be the ``formal disc around $Y$'', let $\dot D_Y:=\spec R((t))$ be the ``punctured formal disc.'' There is commutative diagram of morphisms of $U$-schemes (see~\cite[Sect.~4.1]{FedorovExotic} for details)
\begin{equation*}
\begin{CD}
\dot D_Y @>>> D_Y \\
@VVV @VVV \\
\P^1_U-Y @>>> \P^1_U.
\end{CD}
\end{equation*}
Further, we explained in~\cite{FedorovExotic} that given an $\bH$-bundle over $\P^1_U-Y$, an $\bH$-bundle over $D_Y$, and an isomorphism between their restrictions to $\dot D_Y$, we can glue the bundles to make an $\bH$-bundle over $\P^1_U$; see~\cite[Prop.~4.4]{FedorovExotic}.

This construction can be used to modify $\bH$-bundles over $\P^1_U$ in the following sense. Given an $\bH$-bundle $\cH$ over $\P^1_U$, its trivialization over $\dot D_Y$, and a loop $\alpha\in\bH\bigl(R((t))\bigr)$, we construct a new $\bH$-bundle $\cH(\alpha)$ over $\P^1_U$ as follows. We view~$\alpha$ as an isomorphism between $\cH|_{\dot D_Y}$ and the trivial $\bH$-bundle over $\dot D_Y$, and use it to glue $\cH|_{\P^1_U-Y}$ with the trivial $\bH$-bundle over $D_Y$.

\subsection{End of the proof of Theorem~\ref{th:A1}} We use the notations from the statement of Theorem~\ref{th:A1}. As before, let $U:=\spec R$ and let $Y=0\times U\subset\P^1_U$. Since $\P^1_U-Y\simeq\A^1_U$, we may view $\cF$ as an $\bH$-bundle over $\P^1_U-Y$. Let us trivialize $\cF$ on a complement of a~subscheme $Z\subset\P^1_U-Y$ finite over $U$. Note that $Z$ is closed in $\P^1_U$.
Let us extend $\cF$ to an $\bH$-bundle $\tilde\cF$ over $\P^1_U$ by gluing $\cF$ with the trivial bundle over $\P^1_U-Z$ (observe that both bundles are trivial over the intersection $\P^1_U-Y-Z$).

Consider the $\bH_x$-bundle $\tilde\cF_x$ over $\P^1_x$ obtained by restricting $\tilde\cF$. Note that $\tilde\cF_x$ is generically trivial because it is trivial over $\P^1_x-Z_x$. Thus it is trivial over $\P^1_x-0$ by~\cite[Cor.~3.10(a)]{GilleTorseurs}. Fix such a trivialization.

On the other hand, $\tilde\cF$ is trivialized over $D_Y$, as the morphism $D_Y\to\P^1_U$ factors through $\P^1_U-Z$. Fix such a trivialization, it gives rise to a trivialization of $\tilde\cF_x$ over~$D_{Y_x}$.

Thus we get two trivializations of $\tilde\cF_x$ over $\dot D_{Y_x}$; they differ by an element
\[
    \alpha\in\bH(\dot D_{Y_x})=\bH\bigl(k((t))\bigr),
\]
where $k:=k(x)$.

\begin{lemma}
    The natural map $\bH\bigl(R((t))\bigr)\to\bH\bigl(k((t))\bigr)$ is surjective.
\end{lemma}
\begin{proof}
Let $\bT$ be a split maximal torus in $\bH$. Let $\bB$ be a Borel subgroup scheme such that $\bT\subset\bB\subset\bH$. Let $\bB^-$ be the opposite Borel subgroup scheme (see~\cite[Exp.~XXII, Prop.~5.9.2]{SGA3-3}). Let $\bU^-$ and $\bU$ be the unipotent radicals of $\bB^-$ and $\bB$ respectively. Let $E$ be the subgroup of the abstract group $\bH\bigl(k((t))\bigr)$ generated by $\bU^-\bigl(k((t))\bigr)$ and $\bU\bigl(k((t))\bigr)$. It follows from~\cite[Exp.~XXVI, Cor.~5.2]{SGA3-3} that $\bH\bigl(k((t))\bigr)=\bT\bigl(k((t))\bigr)\cdot E$.

Next, every element of $E$ extends to $\bH\bigl(R((t))\bigr)$, see~\cite[Lemma~5.24]{FedorovPanin}. Thus, it remains to show that every element of $\bT\bigl(k((t))\bigr)$ extends to $\bT\bigl(R((t))\bigr)$. Since $\bT$ is split, it is enough to show that every non-zero element of $k((t))$ extends to an~invertible element of $R((t))$, which is obvious because $R$ is local.
\end{proof}

By the previous lemma, we can extend the loop $\alpha$ to a loop $\tilde\alpha\in\bH\bigl(R((t))\bigr)$. Since $\tilde\cF$ is trivialized over $\dot D_Y$ and $\tilde\alpha^{-1}\in\bH\bigl(R((t))\bigr)$, we obtain a new principal bundle $\tilde\cF(\tilde\alpha^{-1})$ over $\P_U^1$ (see the end of Section~\ref{sect:glue}).

It is easy to see from the construction, that the restriction of $\tilde\cF(\tilde\alpha^{-1})$ to $\P^1_x$ is a trivial $\bH_x$-bundle. Indeed, $\tilde\cF(\tilde\alpha^{-1})|_{\P^1_x}\simeq\tilde\cF_x(\alpha^{-1})$ and $\alpha$ was chosen in such a way that the trivialization of $\tilde\cF_x$ on $\P^1_x-Y_x$ extends to a trivialization of $\tilde\cF_x(\alpha^{-1})$ on $\P^1_x$ (cf.~\cite[Prop.~5.22]{FedorovPanin}).

By~\cite[Cor.~6.12(i)]{ColliotTS1979fibres} there is an embedding $\bH\to\mathbf{GL}(n,U)$. Then by Proposition~\ref{pr:Horrocks}, $\tilde\cF(\tilde\alpha^{-1})$ is isomorphic to a pullback of an~$\bH$-bundle over $U$. Since the restriction of $\tilde\cF(\tilde\alpha^{-1})$ to $Y=0\times U$ is trivial, we see that $\tilde\cF(\tilde\alpha^{-1})$ is trivial. Finally, we see that
\[
    \cF\simeq\tilde\cF(\tilde\alpha^{-1})|_{\P^1_U-(0\times U)}
\]
is trivial. The proof of Theorem~\ref{th:A1} is complete.\qed

\section{Quasi-elementary fibrations}\label{sect:QElFib}
Now we start the proof of Theorem~\ref{th:GrSerre}, which will occupy this and the next sections. In this section we introduce the notion of a quasi-elementary fibration. The main result is Proposition~\ref{pr:QElFib}, which lets us construct a quasi-elementary fibration from the data of Theorem~\ref{th:GrSerre}. We keep the notations from Theorem~\ref{th:GrSerre}. Set $U:=\spec\cO_{X,x}$. We may identify the unique closed point of $U$ with~$x$; denote the residue field of $x$ by $k(x)$. As we have already mentioned, there is a split reductive $\Z$-group scheme $\bG_\Z$ such that $\bG_{X,x}\simeq\bG_\Z\times_\Z\spec\cO_{X,x}$ (see~\cite[Exp.~XXV, Thm.~1.1]{SGA3-3}). Set $\bG:=\bG_\Z\times_\Z\spec\Lambda$; this is a split reductive $\Lambda$-group scheme. Then we have $\bG_{X,x}\simeq\bG\times_\Lambda\spec\cO_{X,x}$. We use this isomorphism to identify the group schemes. Thus, according to our convention, principal $\bG_{X,x}$-bundles are the same as principal $\bG$-bundles.

\subsection{Definition of quasi-elementary fibrations}
The notion of an elementary fibration was introduced in~\cite[Exp.~XI, Def.~3.1]{SGA4-3}. The following notion is a weak version of elementary fibration: we only assume that the projection is smooth over the open part, we do not require the fibers to be integral, and we only require the divisor to be finite surjective over the base (see also~\cite[Def.~2.1]{PaninStavrovaVavilov}).
\begin{definition}\label{def:QEF}
    A \emph{quasi-elementary fibration\/} is an affine smooth morphism of Noetherian schemes $p\colon X'\to S$ that can be included in a commutative diagram
    \[
        \xymatrix{X'\ar[r]^j\ar[dr]_p & \overline X\ar[d]_{\bar p} & Y \ar[l]_i\ar[dl]^q\\
           & S &               }
    \]
    satisfying the following conditions\\
    \stepzero\noindstep $\bar p$ is flat projective of pure relative dimension one;\\
        \noindstep $j$ is an open embedding, $i$ is a closed embedding, and $X'=\overline X-Y$;\\
        \noindstep $\overline X$ is a regular scheme of pure dimension;\\
        \noindstep $q$ is finite surjective.
\end{definition}

We note that $S$ is automatically regular, see~\cite[Prop.~17.3.3(i)]{EGAIV-1}.

\begin{convention}\label{conv:shrinking}
Let $S$ be a scheme, let $T_i$ be $S$-schemes, and let $s\in S$ be a point. By \emph{shrinking $(S,s)$} we mean replacing $S$ by a Zariski neighborhood $S'$ of $s$ and replacing each $T_i$ by $T_i\times_SS'$.
\end{convention}

\subsection{General preliminaries} The following lemma will be used many times, in particular, for constructing quasi-elementary fibrations.
\begin{lemma}\label{lm:finitedivisor}
    Let $T$ and $S$ be Noetherian schemes. Let $\phi\colon T\to S$ be a projective morphism with fibers of dimension one (but not necessarily of pure dimension), let $s\in S$ be a closed point. Let $T_1,T_2\subset T$ be closed subschemes finite over $S$ and such that $T_1\cap T_2=\emptyset$. Then

\stepzero\noindstep\label{FD1}
        If $\cL$ is an $S$-ample line bundle over $T$, then for all large $N$ we may shrink $(S,s)$ so that we can find $\sigma\in H^0(T,\cL^{\otimes N})$ such that $\sigma$ vanishes on $T_1$ and does not vanish at any point of $T_2$.

\noindstep\label{FD2} After shrinking $(S,s)$, we can find a closed subset $D\subset T$ such that $D$ is finite over~$S$, $T_1\subset D$, $T_2\cap D=\emptyset$, and $T-D$ is affine over $S$.

\noindstep\label{FD3} After shrinking $(S,s)$, we can find a finite surjective $S$-morphism $\Pi\colon T\to\P_S^1$ such that $\Pi(T_1)\subset0\times S$, $\Pi(T_2)\subset\infty\times S$.
\end{lemma}
\begin{proof}
For part~\eqref{FD1}, consider $T_0:=T_1\cup(T_2)_s$ and let $I_{T_0}$ be the sheaf of ideals of~$T_0$. Notice that $R^1\phi_*(\cL^{\otimes N}\otimes I_{T_0})$ vanishes in a neighborhood of $s$ for large $N$. Thus, after shrinking $(S,s)$, we can find a section of $\cL^{\otimes N}$ such that this section vanishes on~$T_1$ and does not vanish at any point of $(T_2)_s$. It remains to shrink $(S,s)$ again.

For part~\eqref{FD2}, we may choose a very ample line bundle $\cL$ over $T/S$. By enlarging~$T_2$, we may assume that it contains a closed point in each one-dimensional irreducible component of $T_s$. Let $\sigma$ be a section of $\cL^{\otimes N}$ provided by the first part, let $D$ be the zero locus of $\sigma$. Then the fiber of $D$ over $s$ is finite. Since $D$ is projective over~$S$, dimensions of the fibers are upper semicontinuous (see~\cite[Cor.~13.1.5]{EGAIV-3}). Thus, after shrinking $(S,s)$, we may assume that $D$ is quasi-finite over~$S$. Since $D$ is projective over~$S$, it is finite over~$S$. Now $T-D$ is affine over~$S$ because $\cL^{\otimes N}$ is very ample.

For part~\eqref{FD3}, we may assume that each of $T_1$ and $T_2$ contains at least one point on each irreducible one-dimensional component of $T_s$. Let $\cL$ be a very ample line bundle on $T/S$. Thus, by part~\eqref{FD1}, by shrinking $(S,s)$ and replacing $\cL$ by its power, we can find a section $\tau_1$ of $\cL$ such that $\tau_1$ vanishes on $T_1$ but not at the points of $T_2$. Let $T'$ be the zero set of $\tau_1$.

As in part~\eqref{FD2}, we may assume that $T'$ is finite over $S$. Shrinking~$(S,s)$ and applying part~\eqref{FD1} again, we see that there is a section $\tau_2$ of $\cL^{\otimes N}$ for some $N>0$ such that $\tau_2$ vanishes on $T_2$ but not at the points of $T'$.

Consider the projective morphism $\Pi\colon T\to\P_S^1$ given by $[\tau_1^N:\tau_2]$. Its restriction to~$T_s$ is finite because it is a morphism of one-dimensional projective schemes $T_s\to\P_s^1$ such that both the preimage of zero and the preimage of infinity intersect all one-dimensional components of $T_s$. Thus, by shrinking $(S,s)$, we may assume that $\Pi$ is finite. Clearly, we have $\Pi(T_1)\subset0\times S$ and $\Pi(T_2)\subset\infty\times S$.

It remains to show that $\Pi$ is surjective. Since $\Pi$ is closed (being finite), we only need to check that for any generic point $\omega$ of $S$ the base-changed morphism $\Pi_\omega\colon T_\omega\to\P_\omega^1$ is surjective. If not, then its image is finite, so $\Pi_\omega$ cannot be finite because $T_\omega$ is one-dimensional.  This contradiction completes the proof of surjectivity.
\end{proof}

\subsection{Weighted blow-ups}\label{sect:bl} The scheme $\overline X$ in Definition~\ref{def:QEF} will be constructed via blowing up, similarly, to the Artin's result~\cite[Exp.~XI, Prop.~3.3]{SGA4-3}. However, since Proposition~\ref{pr:Poonen} below produces hypersurfaces rather than hyperplanes, we will need to do a weighted version of blowing up. Denote by $\P_\Z(l_0,\dotsc,l_m)$ the weighted projective space, that is,
\[
    \P_\Z(l_0,\dotsc,l_m):=\mathrm{Proj}(\Z[x_0,\dotsc,x_m]),\qquad\deg x_i=l_i.
\]
For a Noetherian scheme $S$, set $\P_S(l_0,\dotsc,l_m):=\P_\Z(l_0,\dotsc,l_m)\times S$.

Let $Z$ be a reduced Noetherian scheme, let $\cL$ be an invertible sheaf on $Z$ and let
\[
    \sigma_i\in H^0(Z,\cL^{\otimes l_i}),\qquad i=0,\dotsc,m.
\]
Let $Z_0$ be the intersection of the zero loci of $\sigma_i$. The sections $\sigma_i$ give rise to a morphism
\[
    Z-Z_0\xrightarrow{\mu}\P_\Z(l_0,\dotsc,l_m).
\]

Denote by $\Bl_{\sigma_0,\dotsc,\sigma_m}(Z)$ the closure of the graph of $\mu$ in
\[
    \P_Z(l_0,\dotsc,l_m)=Z\times\P_\Z(l_0,\dotsc,l_m).
\]
We call it a~\emph{weighted blow-up of $Z$ along $Z_0$.} We view it as a scheme with reduced scheme structure. We note that it is quite different from the usual blow-up. In particular, it depends on sections $\sigma_0$, \ldots, $\sigma_m$ and not just on their common zero-locus $Z_0$.

We have a projection
\[
    \lambda\colon \Bl_{\sigma_0,\dotsc,\sigma_m}(Z)\to Z;
\]
this is a projective morphism. Note the following easy lemma.

\begin{lemma}\label{lm:bl}
Notation as above, the base change of $\lambda$ with respect to the inclusion $Z-Z_0\hookrightarrow Z$ is an isomorphism.
\end{lemma}
\begin{proof}
Results from the construction.
\end{proof}

We will consider weighted blow-ups only in the case, when $l_0=1$. For a scheme $S$, denote by $\hat\A_S$ the open subset of $\P_S(1,l_1,\dotsc,l_m)$ given by $x_0\ne0$. If $S=\spec A$, then $\hat\A_S=\spec A[y_1,\dotsc,y_m]$, where $y_i=x_i/x_0^{l_i}$. Thus, for any $S$, we have a canonical isomorphism $\hat\A_S=\A_S^m$, in particular, $\hat\A_S$ is smooth over $S$.

In our applications, $Z$ and $Z_0$ will be of pure dimensions and we will have $\dim Z-\dim Z_0=m+1$, in particular, $Z_0$ will be a complete intersection in $Z$.

It is well-known that the blow-up of a regular scheme along a regular closed subscheme is regular (see e.g.~\cite[Thm.~8.1.19]{Liu2002AG}). We prove a slightly weaker statement in the weighted case, which is still sufficient for our purposes.
\begin{lemma}\label{lm:BlReg}
    Notation as above, assume that $Z$ and $Z_0$ are regular schemes and that $Z_0$ is of pure codimension $m+1$ in $Z$. Then $\Bl_{\sigma_0,\dotsc,\sigma_m}(Z)\cap\hat\A_Z$ is a regular scheme. The dimension of $\Bl_{\sigma_0,\dotsc,\sigma_m}(Z)\cap\hat\A_Z$ at a closed point $x$ is equal to the dimension of $Z$ at $\lambda(x)$.
\end{lemma}
\begin{proof}
The statement is local over $Z$, so we may assume that $Z=\spec A$ is affine and that $\cL$ is a trivial line bundle. Choosing a trivialization of $\cL$, we may view $\sigma_i$ as elements of $A$. Put $Z':=\Bl_{\sigma_0,\dotsc,\sigma_m}(Z)\cap\hat\A_Z$. By~\cite[Cor~4.2.17]{Liu2002AG} we only need to show that $Z'$ is regular at each closed point $x$. Let $x\in Z'$ lie over $z:=\lambda(x)\in Z=\spec A$. Viewing $z$ as a prime ideal of $A$, we may assume that $z\supset(\sigma_0,\dotsc,\sigma_m)$ (otherwise we are done by Lemma~\ref{lm:bl}).

\emph{Claim. Let $d\sigma_i$ be the image of $\sigma_i$ in the cotangent space $T^*_z$ of $Z$ at $z$. Then the vectors $d\sigma_0,\dotsc,d\sigma_m$ are linearly independent.}
\begin{proof}
    Consider the local rings $\cO_{Z,z}$ and $\cO_{Z_0,z}=\cO_{Z,z}/(\sigma_0,\dotsc,\sigma_m)$. Since $Z_0$ is of codimension $m+1$ in $Z$, we easily see that the height of the ideal $(\sigma_0,\dotsc,\sigma_m)$ in $\cO_{Z,z}$ is at least $m+1$. Thus
    \begin{equation}\label{eq:DimEstim}
        \dim\cO_{Z,z}\ge\dim\cO_{Z_0,z}+m+1.
    \end{equation}
    Next, the maximal ideal of $\cO_{Z_0,z}$ can be generated by $m_0:=\dim\cO_{Z_0,z}$ elements because $Z_0$ is regular. Let $\tau_1$, \ldots, $\tau_{m_0}$ be such generators. Then the maximal ideal of $\cO_{Z,z}$ is
    \[
       (\sigma_0,\dotsc,\sigma_m,\tilde\tau_1,\dotsc,\tilde\tau_{m_0})
\]
for any lift of $\tau_i$ to $\tilde\tau_i\in\cO_{Z,z}$. If $d\sigma_0,\dotsc,d\sigma_m$ were linearly dependent, we would have $\dim T_z^*<(m+1)+m_0$, which gives a contradiction with~\eqref{eq:DimEstim}.
\end{proof}

Recall that $\hat\A_Z=\spec A[y_1,\dotsc,y_m]$, where $y_i=x_i/x_0^{l_i}$. Thus we have sections $dy_i$ of the cotangent sheaf of $\hat\A_Z$.
Let $T^*_x$ be the cotangent space of $\hat\A_Z$ at $x$. Since $x$ is closed in $Z'$, it is also closed in the $z$-fiber of $Z'$. Thus we can identify $T^*_x=(T^*_z\otimes_{k(z)}k(x))\oplus V$, where $V$ is the $k(x)$-vector space with basis $dy_1$, \ldots, $dy_m$.

Let $T^*$ be the cotangent space of $Z'$ at $x$. We have the surjective projection $T^*_x\to T^*$; denote the kernel of this projection by $K$. Since $\sigma_i-y_i\sigma_0^{l_i}$ vanishes on $Z'$, and $d\sigma_i$ are linearly independent, the dimension of $K$ is at least $m$. Thus
\begin{equation}\label{eq:ineq1}
    \dim T^*\le\dim_z Z.
\end{equation}

We need a general statement about dimensions.

\emph{Claim. Let $f\colon S'\to S$ be a closed and surjective morphism of Noetherian schemes. Let $s\in S'$ be a schematic point. Then $\dim_sS'\ge\dim_{f(s)}S$.}
\begin{proof}
We will prove by induction on $n$ the following statement: \emph{let $f\colon S'\to S$ be a closed and surjective morphism of Noetherian schemes. Let $s\in S'$ be a schematic point. If $\dim_{f(s)}S\ge n$, then $\dim_sS'\ge n$.} If $n=0$, the statement is obvious. Assume that the statement is true for $n=k-1$, let us prove the statement for $n=k$.

Assume that $\dim_{f(s)}S\ge k$. We easily reduce to the case, when $S'$ is integral. We have a chain of prime ideals:
\[
    \cO_{f(s),S}\supsetneq I_{k+1}\supsetneq\dotsb\supsetneq I_1.
\]
The ideal $I_1$ corresponds to the generic point of an integral closed subscheme $S_1\subset S$ such that $f(s)\in S_1$. We have $\dim_{f(s)}S_1\ge k-1$. By induction hypothesis we have $\dim_sf^{-1}(S_1)\ge k-1$. But $f^{-1}(S_1)$ is a proper closed subset of $S'$ containing $s$ and $S'$ is irreducible, so $\dim_sS'\ge k$.
\end{proof}

If follows from this claim that the dimension of $Z'$ at $x$ is greater than or equal to the dimension of $Z$ at $z$ because the morphism $\lambda\colon \Bl_{\sigma_0,\dotsc,\sigma_m}(Z)\to Z$ is closed and surjective (being projective and dominant). Combining this inequality with~\eqref{eq:ineq1}, we see that~$Z'$ is regular at $x$. The statement about the dimensions is now also clear. The proof of Lemma~\ref{lm:BlReg} is complete.
\end{proof}

\subsection{Reductions of principal bundles to Borel subgroups}
Let $\bH$ be a reductive group scheme over a Noetherian scheme $T$. Let $\bB$ be a Borel subgroup scheme of $\bH$ (assumed to exist). Recall that a $\bB$-bundle $\cB$ induces an $\bH$-bundle $\bH\times^{\bB}\cB$. By a~\emph{$\bB$-reduction of an $\bH$-bundle $\cH$\/} we mean a pair $(\cB,s)$, where $\cB$ is a $\bB$-bundle, $s\colon \bH\times^{\bB}\cB\to\cH$ is an isomorphism. If such a reduction exists, we say that $\cH$ can be reduced to $\bB$.

\begin{proposition}\label{pr:codim2}
Let $\bH$ be a reductive group scheme over a Noetherian normal scheme $T$. Let $\bB$ be a Borel subgroup scheme of $\bH$ {\rm(}assumed to exist{\rm)}.
Assume that $\cH$ is an $\bH$-bundle. Then a~$\bB$-reduction of $\cH$ over a dense open subset of $T$ can be extended to an open subset whose complement has codimension at least two in $T$.
\end{proposition}
\begin{proof}
By~\cite[Exp.~XXVI, Cor.~3.6, Lm.~3.20]{SGA3-3} the quotient $\cH/\bB$ is represented by a projective scheme. It is easy to see from the \'etale descent that this quotient classifies $\bB$-reductions of $\cH$. Thus, we just need to show that a section of $\cH/\bB$ over a dense open subset of $T$ can be extended to  an open subset whose complement has codimension at least two in $T$. However, $\cH/\bB$ is proper over $T$, so the statement follows from~\cite[Cor.~7.3.5, Remarque~7.3.7]{EGAII}.
\end{proof}

\subsection{Bertini type Theorems}
Let us define the dimension of the empty scheme to be $-1$. The following proposition follows easily from results of~\cite{PoonenOnBertini}.
\begin{proposition}\label{pr:Poonen}
     Assume that $T_1$, \ldots, $T_n$ are non-empty locally closed subschemes of $\P_k^N$, where $k$ is a finite field. Let $T'\subset T$ be smooth locally closed subschemes of $\P_k^N$, and let~$F$ be a~finite set of closed points of $\P_k^N$. Assume that for all $i$ such that $T_i$ is finite, we have $T_i\cap F=\emptyset$. Then there is a~hypersurface $H\subset\P_k^N$ such that the scheme theoretic intersections $H\cap T$ and $H\cap T'$ are smooth, $F\subset H$, and for $i=1,\dotsc,n$ we have $\dim(H\cap T_i)<\dim T_i$.
\end{proposition}
We remark that if $T_i$ is a point, then the condition $\dim(H\cap T_i)<\dim T_i$ means that $H\cap T_i=\emptyset$.
\begin{proof}
Replacing each $T_i$ by the set of its top-dimensional irreducible components, we may assume that each $T_i$ is irreducible. For $i=1,\dotsc, n$ choose a closed point $p_i\in T_i-F$. We claim that there is a hypersurface $H$ such that
\begin{itemize}
    \item $H\cap T$ and $H\cap T'$ are smooth;
    \item $F\subset H$;
    \item For all $i$, we have $p_i\notin H$.
\end{itemize}
Then the statement follows from~\cite[Thm.~1.3]{PoonenOnBertini}. In more detail, for a point $p\in\P_k^N$, let $\hat\cO_p$ be the completion of the local ring $\cO_{\P_k^N,x}$. We apply~\cite[Thm.~1.3]{PoonenOnBertini} with the following local conditions. If $p=p_i$ for some $i$, then $U_p\subset\hat\cO_p$ is the condition that $H$ does not pass through $p$. If $p\in F$, then the conditions is that $H$ passes through $p$ and the intersections $H\cap T$ and $H\cap T'$ are smooth at $p$ (by definition, an intersection is smooth at $p$, if $p$ does not belong to this intersection). In the remaining case, the condition is just that $H\cap T$ and $H\cap T'$ are smooth at $p$ (cf. the proof of~\cite[Thm.~3.3]{PoonenOnBertini}). The hypersurface $H$ satisfies conditions of the proposition.
\end{proof}

The following proposition will be used to construct quasi-elementary fibrations.
\begin{proposition}\label{pr:GoodNeighb}
Let $k$ be a field and let $X\subset\P_k^{N_1}$ be a closed subscheme of pure dimension $n$, let $X^{sm}\subset X$ be an open subscheme smooth over $k$ and let $x\in X$ be a closed point. Let $n=\dim X$ and Let $T_1$ and $T_2$ be closed subsets of $X$ of dimensions at most $n-1$ and $n-2$ respectively such that $x\notin T_2$. For an integer $r$ consider the $r$-fold Veronese embedding $\P_k^{N_1}\hookrightarrow\P_k^{N_r}$ and identify $X$ with a closed subscheme of $\P_k^{N_r}$, using this embedding. Then there are a positive integer $r$ and sections $\sigma_0\in H^0(\P_k^{N_r},\cO(1))$, $\sigma_1\in H^0(\P_k^{N_r},\cO(l_1))$, \ldots, $\sigma_{n-1}\in H^0(\P_k^{N_r},\cO(l_{n-1}))$ for some positive integers $l_i$ such that

\stepzero\noindstep $\sigma_0(x)\ne0$.

\noindstep\label{GoodNeighb_Transversal} Let $\phi\colon\P^{N_r}_k\dashrightarrow\P_k(1,l_1,\dotsc,l_{n-1})$ be the rational morphism defined by the sections $\sigma_i$. Then the subscheme $\phi^{-1}(\phi(x))\cap X^{sm}$ is smooth of dimension one over $k(\phi(x))$.

\noindstep $\phi^{-1}(\phi(x))\cap T_1$ is finite.

\noindstep $\phi^{-1}(\phi(x))\cap T_2=\emptyset$.

\noindstep\label{GoodNeighb_Transversal2} $\{\sigma_0=\sigma_1=\dotsb=\sigma_{n-1}=0\}\cap X^{sm}$ is finite and \'etale over $k$.

\noindstep $\{\sigma_0=\sigma_1=\dotsb=\sigma_{n-1}=0\}\cap  T_i=\emptyset$ for $i=1,2$.
\end{proposition}

The proof of this proposition in the finite field case is significantly different from the proof in the infinite case.

\begin{proof}[Proof of Proposition~\ref{pr:GoodNeighb} in the case, when $k$ is finite.]
We inductively construct $\sigma_0$, \ldots, $\sigma_{n-1}$ such that $\sigma_0(x)\ne 0$, and for $m=1,\dotsc,n-1$ we have $\sigma_m(x)=0$,
\[
\begin{split}
    \dim(\{\sigma_0=\dotsb=\sigma_m=0\}\cap T_i)&<n-m-i,\\
    \dim(\{\sigma_1=\dotsb=\sigma_m=0\}\cap T_i)&\le n-m-i,
\end{split}
\]
and the intersections
\[
    \{\sigma_0=\dotsb=\sigma_m=0\}\cap X^{sm}\text{ and }\{\sigma_1=\dotsb=\sigma_m=0\}\cap X^{sm}
\]
are smooth over $k$ of dimensions $n-m-1$ and $n-m$ respectively.

For $m=0$ we apply Proposition~\ref{pr:Poonen} with $T_1$, $T_2$, $T_3=x$, $F=T'=\emptyset$, and $T=X^{sm}$. We get $\sigma_0\in H^0(\P_k^{N_1},\cO(r))$. We can view it as an element of $H^0(\P_k^{N_r},\cO(1))$. We have $\sigma_0(x)\ne0$, $\{\sigma_0=0\}\cap X^{sm}$ is smooth of dimension $n-1$ over $k$, and $\dim(\{\sigma_0=0\}\cap T_i)<n-i$.

Assume that $\sigma_0$, \ldots, $\sigma_{m-1}$ are already constructed. To construct $\sigma_m$ we apply Proposition~\ref{pr:Poonen} to
$\{\sigma_0=\dotsb=\sigma_{m-1}=0\}\cap T_i$, $T'=\{\sigma_0=\dotsb=\sigma_{m-1}=0\}\cap X^{sm}$, $T=\{\sigma_1=\dotsb=\sigma_{m-1}=0\}\cap X^{sm}$, and $F=\{x\}$.

By construction, $\sigma_0$, \ldots, $\sigma_{n-1}$ satisfy the conditions of the proposition. (Note that $\phi^{-1}(\phi(x))$ is contained in $\{\sigma_1=\ldots=\sigma_{n-1}=0\}$).
\end{proof}

\begin{proof}[Proof of Proposition~\ref{pr:GoodNeighb} in the case, when $k$ is infinite.]
We will take $r=2$ and $l_1=\dotsb=l_{n-1}=1$. Set $V:=H^0(\P_k^{N_2},\cO(1))$; we view the vector space $V$ as a scheme over $k$.

\begin{lemma}
Assume that $k$ is algebraically closed. Let $X$, $X^{sm}$, $x$, $T_1$, $T_2$ be as in the statement of Proposition~\ref{pr:GoodNeighb}. Then there is a non-empty open subset $W\subset V^n$ such that every point $(\sigma_0,\dotsc,\sigma_{n-1})\in W$ satisfies the conditions of the proposition. More precisely, conditions~\eqref{GoodNeighb_Transversal} and~\eqref{GoodNeighb_Transversal2} mean that $\phi^{-1}(\phi(x))$ and $\{\sigma_0=\dotsb=\sigma_{n-1}=0\}$ intersect $X^{sm}$ transversally.
\end{lemma}
\begin{proof}
Recall that we have a 2-fold Veronese embedding $\P_k^{N_1}\hookrightarrow\P_k^{N_2}$. Let $\Gr_x(N_2+1,n)$ stand for the Grassmannian of codimension $n-1$ subspaces containing $x$ in $\P^{N_2}_k$. It follows from~\cite[Exp.~XI, Thm.~2.1(ii)]{SGA4-3} that there is a non-empty open subset $U\subset\Gr_x(N_2+1,n)$ such that every subspace from $U$ intersects $X^{sm}$ transversally. Let $W'\subset V^n$ be the open subspace defined by the conditions that for all $(\sigma_0,\dotsc,\sigma_{n-1})\in W_i'$ we have $\sigma_0(x)\ne0$, and the rational morphism $\phi\colon\P^{N_2}_k\dashrightarrow\P_k^{n-1}$ defined by $\sigma_0,\dotsc,\sigma_{n-1}$ is such that the Zariski closure of $\phi^{-1}(\phi(x))$ is a codimension $n-1$ projective subspace of $\P_k^{N_2}$. We get a morphism $\pi:W'\to\Gr_x(N_2+1,n)$ sending $(\sigma_0,\dotsc,\sigma_{n-1})\in W'$ to the Zariski closure of $\phi(\phi^{-1}(x))$. Set $W'':=\pi^{-1}(U)$. Clearly, $W''$ is open. Since $\pi$ is dominant, $W''$ is non-empty. For each  $(\sigma_0,\dotsc,\sigma_{n-1})\in W''$ the intersection $\phi^{-1}(\phi(x))\cap X^{sm}$ is transversal. We show similarly that for every other condition of the proposition there is a non-empty open subset in $V^n$ whose points possess the property (use~\cite[Exp.~XI, Thm.~2.1(i)]{SGA4-3} for~\eqref{GoodNeighb_Transversal2}, use an inductive argument similar to the proof in the finite field case for the other properties). The intersection of these open subsets is non-empty, since $V^n$ is irreducible; it is the required set $W$.
\end{proof}

We return to the proof of Proposition~\ref{pr:GoodNeighb}. Let $\bar k$ be an algebraic closure of $k$. Consider the finite scheme $x\times_k\bar k$. Let $x_1,\dotsc,x_m\in X_{\bar k}$ be all its closed points. Applying the previous lemma to $X_{\bar k}$, $X^{sm}_{\bar k}$, $x_i$, $(T_1)_{\bar k}$, and $(T_2)_{\bar k}$, we get a dense open subset $\tilde W_i\subset\tilde V:=H^0(\P^{N_2}_{\bar k},\cO(1))$. Let $p\colon\tilde V\to V$ be the projection. There is a non-empty open subset $W\subset V$ such that $p^{-1}(W)\subset\cap_{i=1}^m\tilde W_i$. Since $k$ is infinite, we can find a $k$-rational point $(\sigma_0,\dotsc,\sigma_{n-1})\in W$. We claim that this point satisfies the conditions of the proposition.

Let us check condition~\eqref{GoodNeighb_Transversal} first. Set $s:=\phi(x)$, $k':=k(s)$, $Y:=\phi^{-1}(s)$. Then~$Y$ is smooth over $s$. We need a lemma whose meaning is that $Y$ and $X^{sm}$ intersect transversally.
\begin{lemma}
Let $Z$ be the scheme theoretic intersection of $X^{sm}$ and $Y$. Then the canonical morphism of sheaves
\[
    \alpha\colon\Omega_{\P^{N_2}_k/k}|_Z\to\Omega_{Y/s}|_Z\oplus\Omega_{X^{sm}/k}|_Z
\]
is injective and locally split.
\end{lemma}
\begin{proof}
Note that the sheaves $\Omega_{\P^{N_2}_k/k}|_Z$ and $\Omega_{Y/s}|_Z\oplus\Omega_{X^{sm}/k}|_Z$ are locally free. Thus, it is enough to check that for every geometric  point $\bar z\in Z(\bar k)$ the pullback of this morphism to $\bar z$ is injective. Let $\bar\phi\colon\P^{N_2}_{\bar k}\dashrightarrow\P_{\bar k}^{n-1}$ be the rational morphism defined by the sections $\bar\sigma_i:=\sigma_i\times_k\bar k$. Set $\bar x:=x\times_k\spec\bar k$, $\bar s:=s\times_k\spec\bar k$. Set $\overline Y_i:=\bar\phi^{-1}(\bar\phi(x_i))$ and $\overline X^{sm}:=X^{sm}\times_k\spec\bar k$.
Note that $\phi(\bar z)=\phi(x_i)$ for some $i$. The pullback of $\alpha$ to $\bar z$ is the canonical morphism
\[
    \bar\alpha\colon\Omega_{\P^{N_2}_{\bar k}/{\bar k}}|_{\bar z}\to\Omega_{(Y\times_k\spec\bar k)/\bar s}|_{\bar z}\oplus\Omega_{\overline X^{sm}/\bar k}|_{\bar z}.
\]
We also have a canonical isomorphism
\[
    \Omega_{(Y\times_k\spec\bar k)/\bar s}|_{\bar z}\oplus\Omega_{\overline X^{sm}/\bar k}|_{\bar z}\xrightarrow{\simeq}\Omega_{\overline Y_i/\bar k}|_{\bar z}\oplus\Omega_{\overline X^{sm}/\bar k}|_{\bar z}.
\]
Indeed, $\overline Y_i$ is the fiber of $Y\times_k\spec\bar k$ over $\phi(\bar z)=\phi(x_i)\in\bar s$. The composition of these two morphisms is injective because $\overline X^{sm}$ and $\overline Y_i$ intersect transversally. It follow that~$\bar\alpha$ is injective.
\end{proof}

If follows from the above lemma that the cokernel of $\alpha$ is a locally free sheaf of rank one. On the other hand, we have a surjective morphism $\Coker\alpha\twoheadrightarrow\Omega_{Z/k'}$. We will use the following lemma.
\begin{lemma}
Let $Z$ be a scheme of finite type over a field $k'$ such that all its components are of dimension at least one and the sheaf $\Omega_{Z/k'}$ is locally generated by one element. Then $Z$ is smooth over $k'$.
\end{lemma}
\begin{proof}
It is enough to prove the statement after a base change to an algebraic closure of $k'$. Thus, we may assume that $k'$ is algebraically closed. In this case, for every closed point of $Z$ we have $(\Omega_{Z/k'})_z=\fm_z/\fm_z^2$, where $\fm_z$ is the maximal ideal of the local ring $\cO_{Z,z}$. Since this vector space is generated by one element, we see that $Z$ is regular at $z$. Thus $Z$ is regular. Since $k'$ is algebraically closed, we see that $Z$ is smooth over~$k'$.
\end{proof}
Now condition~\eqref{GoodNeighb_Transversal} follows from the above lemma. Condition~\eqref{GoodNeighb_Transversal2} is verified similarly. The remaining conditions are clear.
\end{proof}

\subsection{Constructing quasi-elementary fibrations} In this section we prove the following proposition.

\begin{proposition}\label{pr:ConstrQEF}
Let $X\to\spec\Lambda$ and $x\in X$ be as in Theorem~\ref{th:GrSerre}. That is, $\Lambda$ is an excellent discrete valuation ring, $b\in\spec\Lambda$ is the closed point. Also, $X$ is an integral scheme, $X$ is flat and projective over $\spec\Lambda$, and $X$ satisfies conditions~(I) and~(II) of Section~\ref{sect:MainRes}. The projection $X\to\spec\Lambda$ is smooth at the closed point $x\in X$. We also assume that the relative dimension of the flat morphism $X\to\spec\Lambda$ is at least one. Let $X^0$ be an open subscheme of $X$ such that $x\in X^0$. Assume also that the intersection of $X^0$ with the fiber $X_b$ is dense in this fiber. Assume that $Z$ is a closed subset of $X^0$ of codimension at least two. Then there is an open subscheme $X'\subset X^0$ containing $x$, a connected $\Lambda$-scheme $S$ smooth over $\Lambda$, and a $\Lambda$-morphism $p\colon X'\to S$ such that $p$ is a quasi-elementary fibration and $Z\cap X'$ is finite over $S$.
\end{proposition}
\begin{proof}
The proof is somewhat technical but it follows the same strategy as the proofs of~\cite[Prop.~2.3]{PaninStavrovaVavilov} and of Artin's result~\cite[Exp.~XI, Prop.~3.3]{SGA4-3}.

We may assume that $X^0$ is smooth over $\spec\Lambda$ (use condition~(I) and openness of smoothness). Set $Y^0:=X-X^0$. Set $n=\dim X-1=\dim X_b$. Note that $\dim Y^0_b\le n-1$. Denote by $\overline Z$ the Zariski closure of $Z$ in $X$. Then $(\overline Z)_b$ is the intersection of $\overline Z$ with $X_b$, which is in general larger than the closure of $Z_b$. In any case,
\[
    \dim(\overline Z)_b\le\dim\overline Z\le n-1.
\]

\begin{lemma}\label{lm:Bertini}
    There is a $\Lambda$-embedding $X\hookrightarrow\P^N_\Lambda$ for some $N$, a section $\sigma_0\in H^0(\P^N_{k(b)},\cO(1))$, and sections $\sigma_i\in H^0(\P^N_{k(b)},\cO(l_i))$ for some positive integers $l_i$, satisfying the following conditions

    \begin{itemize}
    \item $\sigma_0(x)\ne0$;
    \item $(X^{sing})_b\cap\phi^{-1}(\phi(x))=\emptyset$, where $\phi\colon\P^N_k\dashrightarrow\P_k(1,l_1,\dotsc,l_{n-1})$ is the rational morphism defined by the sections $\sigma_i$;
    \item $Y^0_b\cap\phi^{-1}(\phi(x))$ is finite;
    \item $(\overline Z)_b\cap\phi^{-1}(\phi(x))$ is finite;
    \item $(\overline Z-Z)_b\cap\phi^{-1}(\phi(x))=\emptyset$;
    \item $(\overline Z)_b\cap\{\sigma_0=\sigma_1=\dotsb=\sigma_{n-1}=0\}=\emptyset$;
    \item $X^0_b\cap\phi^{-1}(\phi(x))$ is smooth of dimension one over $b$;
    \item $(X^{sing})_b\cap\{\sigma_0=\sigma_1=\dotsb=\sigma_{n-1}=0\}=\emptyset$.
    \item $X_b\cap\{\sigma_0=\sigma_1=\dotsb=\sigma_{n-1}=0\}$ is finite and \'etale over $b$;
    \end{itemize}
\end{lemma}

\begin{proof}
Note that
\begin{itemize}
\item $\dim (X^{sing})_b\le n-2$ by condition (II) on $X$;
\item $x\notin X^{sing}$ by assumption;
\item $\dim(\overline Z-Z)_b\le\dim(\overline Z-Z)\le n-2$;
\item $x\notin\overline Z-Z$ because $Z$ is closed in $X_b^0$ and $x\in X_b^0$.
\end{itemize}

Consider any $\Lambda$-embedding $X\hookrightarrow\P^{N_1}_\Lambda$ for an integer $N_1$. We apply Proposition~\ref{pr:GoodNeighb} with $X_b$, $X_b^0$, $x$, $T_1=Y_b^0\cup(\overline Z)_b$ and $T_2=(\overline Z-Z)_b\cup(X^{sing})_b$.
We claim that the composition of $X\hookrightarrow\P^{N_1}_\Lambda$ with the $r$-fold Veronese embedding satisfies the requirements. In fact, all conditions except the last one are immediate. Since $\{\sigma_0=\sigma_1=\dotsc=\sigma_{n-1}=0\}\cap Y_b^0=\emptyset$, we get the last condition.
\end{proof}

We can lift each $\sigma_i$ to a section $\tilde\sigma_i\in H^0(\P_\Lambda^N,\cO(l_i))$, because the reduction map from the $\Lambda$-module $H^0(\P_\Lambda^N,\cO(l_i))$ to the $k(b)$-vector space $H^0(\P_{k(b)}^N,\cO(l_i))$ is surjective. Similarly, we can lift $\sigma_0$ to a section $\tilde\sigma_0\in H^0(\P_\Lambda^N,\cO(1))$. Set $\cL:=\cO_{\P_\Lambda^N}(1)|_X$, $\sigma'_i:=\tilde\sigma_i|_X$, so that
$\sigma'_i\in H^0(X,\cL^{\otimes l_i})$. Set $X_{\Bl}:=\Bl_{\sigma'_0,\dotsc,\sigma'_{n-1}}(X)$ (see Section~\ref{sect:bl}).

Denote by $\tilde\phi$ the rational morphism $\P_\Lambda^N\dashrightarrow\P_\Lambda(1,l_1,\dotsc,l_{n-1})$ defined by $\tilde\sigma_0$, \ldots, $\tilde\sigma_{n-1}$. Let $\lambda\colon X_{\Bl}\to X$ be the canonical morphism. Denote by $E$ the exceptional locus of $\lambda$, that is, $E=\lambda^{-1}(X\cap \{\tilde\sigma_0=\dotsb=\tilde\sigma_{n-1}=0\})$. By Lemma~\ref{lm:bl}, $\lambda$ induces an isomorphism $X_{\Bl}-E=X-\{\tilde\sigma_0=\dotsb=\tilde\sigma_{n-1}=0\})$. Set $\hat Z:=\lambda^{-1}(\overline Z)$ and $\hat Y:=\lambda^{-1}(Y^0)$. We identify $x$ with its unique $\lambda$-preimage in $X_{\Bl}$, see Lemma~\ref{lm:bl}.

We have a projective morphism $\bar p\colon X_{\Bl}\to S:=\P_\Lambda(1,l_1,\dotsc,l_{n-1})$, defined as the composition of the closed embedding $X_{\Bl}\to\P_X(1,l_1,\dotsc,l_{n-1})$ and the projection $\P_X(1,l_1,\dotsc,l_{n-1})\to\P_\Lambda(1,l_1,\dotsc,l_{n-1})$. Set $s:=\bar p(x)=\tilde\phi(x)$ and $F:=\bar p^{-1}(s)$.

\begin{lemma}\label{lm:Step3}
Denote by $X_{\Bl}^s$ the set of points of $X_{\Bl}$, where $\bar p$ is not smooth. Then
\stepzero
\noindstep\label{F0} $F$ is the Zariski closure of $\lambda^{-1}(\tilde\phi^{-1}(s)\cap X^0)$;\\
\noindstep\label{F1} $F$ is of pure dimension one;\\
\noindstep\label{F2} $X_{\Bl}$ is regular at the points of $F$;\\
\noindstep\label{F3} $\bar p$ is flat at the points of $F$;\\
\noindstep\label{F4} $X_{\Bl}^s\cap F$ is finite;\\
\noindstep\label{F5} $E\cap F$, $\hat Y\cap F$, and $\hat Z\cap F$ are finite;\\
\noindstep\label{F6} $\hat Z\cap F=\lambda^{-1}(Z)\cap F$;\\
\noindstep\label{F7} $\hat Z\cap E\cap F=\hat Z\cap\hat Y\cap F=\hat Z\cap X_{\Bl}^s\cap F=\emptyset$.
\end{lemma}
\begin{proof}
According to Lemma~\ref{lm:Bertini}, $\{\sigma_0=\dots=\sigma_{n-1}=0\}$ is finite over $b$. It follows that $\bar p|_E$ is a finite morphism. Using this fact and the fact that $\lambda^{-1}(\tilde\phi^{-1}(s)\cap Y^0)$ is finite, we see that $F$ is a union of $\lambda^{-1}(\tilde\phi^{-1}(s)\cap X^0)$ and a finite set. However, since $X_{\Bl}$ is irreducible of dimension $n+1$ all the components of $F$ have dimension at least one. Now~\eqref{F0} follows and~\eqref{F1} follows from~\eqref{F0}.

Next, we have a \emph{regular\/} open subscheme $X-X^{sing}\subset X$. Set
\[
    \tilde L:=\{\tilde\sigma_0=\dotsb=\tilde\sigma_{n-1}=0\}.
\]
We claim that $X\cap\tilde L$ is contained in $X-X^{sing}$. Indeed, the intersection $X^{sing}\cap\tilde L$ is proper over $\Lambda$, so, if it is nonempty, it must intersect the closed fiber, which contradicts the penultimate statement in Lemma~\ref{lm:Bertini}.

Further, we claim that $X\cap\tilde L$ is a locally complete intersection in $X-X^{sing}$. Indeed, the integral scheme $X-X^{sing}$ is of dimension $n+1$ because the closed fiber $X_b-X_b^{sing}$ is of dimension $n$, and $X-X^{sing}$ is flat over $\Lambda$. Further, $X\cap\tilde L$ is locally given by $n$ equations. So to show that it is a locally complete intersection we just need to show that every component of $X\cap\tilde L$ has dimension at most one. Again, the morphism $X\cap\tilde L\to\spec\Lambda$ is proper, so it is enough to show that the central fiber $X_b\cap\tilde L$ is finite, which is a part of Lemma~\ref{lm:Bertini}. We see that $X\cap\tilde L$ is a locally complete intersection in the regular scheme $X-X^{sing}$. Now by~\cite[Thm.~23.1]{MatsumuraCommRingTh} $X\cap\tilde L$ is flat over $\spec\Lambda$.
It follows from Lemma~\ref{lm:Bertini} and openness of \'etalness for flat morphisms that $X\cap\tilde L$ is \'etale over $\spec\Lambda$.
Since
\[
    \lambda(F)\subset\tilde\phi^{-1}(s)\cup(X_b\cap\tilde L)\subset X-X^{sing},
\]
part~\eqref{F2} follows from Lemma~\ref{lm:BlReg}.

Next, let $x$ be a closed point of $F$. Then $\lambda(x)$ is a closed point of $X$ because $\lambda$ is proper. Applying Lemma~\ref{lm:BlReg}, we see that we have
\[
    \dim_x X_{\Bl}=\dim_{\lambda(x)}X=n+1=\dim_{\bar p(x)} S+\dim_x F,
\]
where we used part~\eqref{F1}. Thus $\bar p$ is flat at $x$ by~\cite[Thm.~23.1]{MatsumuraCommRingTh} and part~\eqref{F2}. Now~\eqref{F3} follows because the set of points, where $\bar p$ is flat is open, and closed points are dense in $F$ because $F$ is a scheme of finite type over a field.

To prove~\eqref{F4} note that $\bar p$, being flat on $F$, is smooth exactly where the fiber is smooth. Now use part~\eqref{F0} and Lemma~\ref{lm:Bertini}. The remaining statements follow from~\eqref{F0} and the respective properties of $L$ and~$H_0$ (see Lemma~\ref{lm:Bertini}).
\end{proof}

\begin{lemma}\label{lm:Step4}
After shrinking $(S,1_b)$ in the sense of Convention~\ref{conv:shrinking} and replacing $X_{\Bl}$, $E$, $\hat Y$, and $\hat Z$ by their intersections with $\bar p^{-1}(S)$, we may assume that\\
\stepzero
    \noindstep\label{G1} $S$ is connected, affine, and smooth over $\Lambda$;\\
    \noindstep\label{G2} $X_{\Bl}$ is regular;\\
    \noindstep\label{G3} $\bar p$ is flat of pure relative dimension one;\\
    \noindstep\label{G4} $X_{\Bl}^s$, $E$, $\hat Y$, and $\hat Z$ are finite over $S$;\\
    \noindstep\label{G5} There is a closed subset $Y\subset X_{\Bl}$ finite and surjective over $S$ such that $Y\supset E\cup X_{\Bl}^s\cup\hat Y$, $Y\cap\hat Z=\emptyset$, $x\notin Y$, and $X_{\Bl}-Y$ is affine.
\end{lemma}
\begin{proof}
First of all,~\eqref{G1} is obvious,~\eqref{G2} follows from the fact that the set of points, where $X_{\Bl}$ is regular is open in $X_{\Bl}$ (because $\Lambda$ is excellent) and the fact that $\bar p$ is closed. Next, flatness in~\eqref{G3} follows from Lemma~\ref{lm:Step3}\eqref{F3} together with the fact that set of points, where a morphism of finite type is flat, is open.

Further, it follows from the construction that $X_{\Bl}$ and $S$ are irreducible. Since $\bar p$ is flat, it is open, hence we can apply~\cite[Cor.~14.2.2.(i)]{EGAIV-3} to conclude that $\bar p$ is equidimensional. The set $F=\bar p^{-1}(1_b)$ is of pure dimension one by Lemma~\ref{lm:Step3}\eqref{F1}. We see that $\bar p$ is of pure relative dimension one.

Next,~\eqref{G4} follows because the dimensions of fibers of a~projective morphism are upper semicontinuous (see~\cite[Cor.~13.1.5]{EGAIV-3}) and a quasi-finite projective morphism is finite; finally,~\eqref{G5} follows from Lemma~\ref{lm:finitedivisor}\eqref{FD2} (note that $Y$ is automatically surjective over $S$ because the fibers of $X_{\Bl}$ are projective, while the fibers of $X_{\Bl}-Y$ are affine).
\end{proof}

Let us summarize. Just before Lemma~\ref{lm:Step4} we constructed a projective morphism $\bar p\colon X_{\Bl}\to S$, a morphism $\lambda\colon X_{\Bl}\to X$ and a subscheme $\hat Y\subset X_{\Bl}$. Then in Lemma~\ref{lm:Step4} we replaced $S$, $X_{\Bl}$, and $\hat Y$ by open subschemes following Convention~\ref{conv:shrinking}. We also constructed a closed subset $Y\subset X_{\Bl}$. The restriction of $\lambda$ to $X':=X_{\Bl}-Y$ is an open embedding, so we can identify $X'$ with an open subset of $X^0$. Now it follows from the construction and Lemma~\ref{lm:Step4} that $\bar p|_{X'}\colon X'\to S$ is a quasi-elementary fibration (with $\overline X=X_{\Bl}$). Also, shrinking $(S,1_b)$ again if necessary, we may assume that under the identification of $X'$ and $\lambda(X')$ we have $\hat Z=Z\cap X'$, so $Z\cap X'$ is finite over $S$ (use Lemma~\ref{lm:Step4}\eqref{F6}). This completes the proof of Proposition~\ref{pr:ConstrQEF}.
\end{proof}

\begin{proposition}\label{pr:QElFib}
Let $\Lambda$, $X$, $x$ and $\bG_{X,x}$ be as in Theorem~\ref{th:GrSerre} and let $\bG$ be a split $\Lambda$-group scheme such that $\bG_{X,x}\simeq\bG\times_\Lambda\spec\cO_{X,x}$. Assume also that the relative dimension of $X\to\spec\Lambda$ is at least one.  Let $\cG$ be a $\bG_{X,x}$-bundle having a rational section. Then there are
\begin{itemize}
\item an open affine subscheme $X'\subset X$ containing $x$;
\item a quasi-elementary fibration $p\colon X'\to S$ with $S$ connected and smooth over~$\Lambda$;
\item a principal divisor $Z'\subset X'$ finite over $S$;
\item a $\bG$-bundle $\cF$ over $X'$ extending $\cG$ such that $\cF$ is trivial over $X'-Z'$;
\item a finite surjective $S$-morphism $X'\to\A^1_S$.
\end{itemize}
\end{proposition}

\subsection{Proof of Proposition~\ref{pr:QElFib}} We will use the notations and the assumptions from the statement of the proposition. We start with a lemma.

\begin{lemma}
We can find a regular open subscheme $X^0\subset X$ such that $x\in X^0$, $X^0\cap X_b$ is dense in $X_b$, and $\cG$ can be extended to a $\bG$-bundle $\cG^0$ over $X^0$ that is trivial over a dense open subset of $X^0$.
\end{lemma}
\begin{proof}
We can find an open subscheme $X_1\subset X$ such that $x\in X_1$, and $\cG$ can be extended to a $\bG$-bundle $\cG_1$ over $X_1$. Since $\cG$ is generically trivial, $\cG_1$ is trivial on the complement of a proper closed subscheme $Z_1\subset X_1$. Since $X_b$ is smooth at $x$, we see that $x$ lies on a single irreducible component of $X_b$. Thus we may assume that $X_1$ does not intersect irreducible components of $X_b$ other than that containing~$x$.

Denote by $n$ the dimension of $X_b$. It follows from the flatness of the morphism $\pi\colon X\to\spec\Lambda$ that the Krull dimension of $X$ is $n+1$. Let $\overline Z_1$ be the Zariski closure of $Z_1$ in $X$. We have
\[
\dim(\overline Z_1-Z_1)<\dim Z_1\le n.
\]
It follows that $\overline Z_1-Z_1$ cannot contain an irreducible component of $X_b$ (use flatness of $\pi\colon X\to\spec\Lambda$ again). Thus $\overline Z_1$ cannot contain irreducible components of $X_b$ other than the component containing~$x$. Consider the trivial $\bG$-bundle $\cG_{triv}$ over $X-\overline Z_1$. The trivialization of $\cG_1$ is an isomorphism between $\cG_1$ and $\cG_{triv}$ over the open subset $X_1-Z_1$. Thus we can glue $\cG_1$ with $\cG_{triv}$ over $X_1-Z_1$ to make a $\bG$-bundle $\cG_2$ over $X_2:=(X-\overline Z_1)\cup X_1$. One now takes $X^0$ to be the regular locus of $X_2$ and sets $\cG^0:=\cG_2|_{X^0}$. It follows from the construction and property (II) of $\pi\colon X\to\spec\Lambda$, that $X^0$ satisfies the requirements of the lemma.
\end{proof}

Fix such $X^0$ and $\cG^0$ provided by the above lemma. Since $\bG$ is split, there is a split maximal torus $\bT\subset\bG$ and a Borel subgroup $\bB\subset\bG$ containing $\bT$. Fix such $\bT\subset\bB$. The trivialization of $\cG^0$ over a dense open subset of $X^0$ gives a~$\bB$-reduction of~$\cG^0$ over this subset. Thus, according to Proposition~\ref{pr:codim2}, $\cG^0$ can be reduced to $\bB$ over $X^0-Z$, where $Z$ is closed and of codimension at least two in $X^0$. By Proposition~\ref{pr:ConstrQEF}, there is an open subscheme $X'\subset X^0$ containing $x$, and a~quasi-elementary fibration $p\colon X'\to S$ with $S$ connected and smooth over $\spec\Lambda$ such that $Z\cap X'$ is finite over~$S$. We may assume that $S$ is affine. We will use the notations from Definition~\ref{def:QEF}. In particular, we have a flat projective morphism $\bar p\colon \overline X\to S$. Set $s:=\bar p(x)$ and $F:=\bar p^{-1}(s)$.

Note that $Z\cap X'$ is closed in $\overline X$ (being finite over $S$), so applying Lemma~\ref{lm:finitedivisor}\eqref{FD2} to $Z\cap X',Y\subset\overline X$, we find a closed subscheme $Z_1\subset X'$ such that $Z\cap X'\subset Z_1$, $Z_1$ is finite over $S$, and $\overline X-Z_1$ is an affine scheme (we might need to shrink $(S,s)$). Then $X'-Z_1=(\overline X-Z_1)\cap X'$ is also affine as the intersection of two open affine subschemes of a separated scheme.

Set $\cF:=\cG^0|_{X'}$. Note that $\cF$ is reduced to the Borel subgroup $\bB$ over $X'-Z_1$, that is, there is a $\bB$-bundle $\cB$ over $X'-Z_1$ such that the $\bG$-bundles $\bG\times^{\bB}\cB$ and $\cF|_{X'-Z_1}$ are isomorphic. Let $\bU$ be the unipotent radical of $\bB$, then the quotient $\bB/\bU$ is isomorphic to the split torus $\bT$. Let $\cB/\bU$ be the induced $\bT$-bundle (this quotient is representable by a scheme because of the \'etale descent, see~\cite[VIII, Cor.~7.9]{SGA1} for a stronger statement). We claim that (after shrinking $(S,s)$ again) we can find a closed subset $Z_2\subset X'-Z_1$ such that $Z_2$ is finite over $S$, the bundle $\cB/\bU$ is trivial over $X'-Z_1-Z_2$, and $X'-Z_1-Z_2$ is affine. Since a principal bundle for a split torus corresponds to a collection of line bundles, and the intersection of open affine subschemes of a separated scheme is affine, this follows from the next lemma.

\begin{lemma}
Let $\ell$ be a line bundle over $X'':=X'-Z_1$. Then {\rm(}after shrinking $(S,s)${\rm)} there is a subscheme $Z''\subset X''$ finite over $S$ such that $\ell$ is trivial over $X''-Z''$ and $X''-Z''$ is affine.
\end{lemma}
\begin{proof}
First of all, we may extend $\ell$ to $\overline X$ because $\overline X$ is a regular scheme. Set $X_\infty:=(\overline X-X'')\cap F$, this is a finite scheme. Adding finitely many points to $X_\infty$, we may assume that it intersects each irreducible component of $F$. Since $\overline X$ is projective over an affine scheme, $X_\infty$ is contained in an open affine subscheme of $\overline X$. Thus we can consider the semilocal ring of $X_\infty$ in $\overline X$; denote it by $A$. Since $A$ is semilocal, $\ell$ is trivial over $A$. Thus there is a closed subscheme $Z''\subset\overline X$ such that $\ell|_{\overline X-Z''}$ is trivial and $Z''\cap X_\infty=\emptyset$. In particular, $Z''\cap F$ is finite by our choice of $X_\infty$. Shrinking $(S,s)$, we may assume that $Z''$ is finite over $S$ and that $Z''\subset X''$. Now by Lemma~\ref{lm:finitedivisor}\eqref{FD2} we may assume that $\overline X-Z''$ (and thus $X''-Z''$) are affine.
\end{proof}

Now we finish the proof of the proposition. Choose $Z_2\subset X'-Z_1$ such that $\cB/\bU$ is trivial over $X'-Z_1-Z_2$, $X'-Z_1-Z_2$ is affine, and $Z_2$ is finite over $S$. By~\cite[Exp.~XXVI, Cor.~2.3]{SGA3-3} we see that $\cB$ and thus $\cF$ are trivial over $X'-Z_1-Z_2$.

Note that $Z_1\cup Z_2$ is closed in $\overline X$. By Lemma~\ref{lm:finitedivisor}\eqref{FD3}, by shrinking $(S,s)$ we can find a~finite surjective morphism $\Pi\colon \overline X\to\P_S^1$ such that
\[
    Z_1\cup Z_2\cup\{x\}\subset Z':=\Pi^{-1}(0\times S),\quad Y\subset Y':=\Pi^{-1}(\infty\times S).
\]
Clearly, $X''':=\overline X-Y'$ is smooth and affine over $S$. Also,~$Z'$ is finite over $S$. It is easy to check that the restriction of $p$ to $X'''$ is a quasi-elementary fibration. Next,~$Z'$ is a principal divisor in $X'''$ because $0\times S$ is a principal divisor in $\A^1_S$. Clearly, $\cF$ is trivial over $X'''-Z'$. This completes the proof of Proposition~\ref{pr:QElFib}.\qed

\section{End of Proof of Theorem~\ref{th:GrSerre}}\label{sect:EndOfProof}
In this section, we use the notion of a nice triple to reduce Theorem~\ref{th:GrSerre} to Theorem~\ref{th:A1}. We keep the notations and the assumptions from Theorem~\ref{th:GrSerre}. As before, $U:=\spec\cO_{X,x}$ and $\bG$ is a $\Lambda$-group scheme such that $\bG_{X,x}=\bG\times_{\cO_{X,x}}\spec\Lambda$. Let $\cG$ be a generically trivial $\bG$-bundle over $U$. We need to show that $\cG$ is trivial. By~\cite{Nisnevich1} we may assume that the relative dimension of the flat morphism $X\to\spec\Lambda$ is at least one (though, in fact, it is easy to prove the theorem if this dimension is zero, see Remark~\ref{rm:NoNisnevich}).

\subsection{Nice triples}
Recall the notion of a nice triple from~\cite[Def.~3.1]{PaninStavrovaVavilov}.

\begin{definition}\label{def:nice}
A~\emph{nice triple} over $U$ is a triple $(q_U\colon \cX\to U,f,\Delta)$, where $\cX$ is an irreducible affine scheme smooth over $U$ and such that all its fibers are of pure dimension one, $f\in\Gamma(\cX,\cO_{\cX})$ is such that its zero locus $\cZ$ is finite over $U$, and $\Delta\colon U\to\cX$ is a~section of $q_U$ such that $\Delta^*(f)\ne0$. These data are subject to the condition that there exists a finite $U$-morphism $\cX\to\A^1_U$.
\end{definition}

\begin{remark}
The finiteness of $\cZ$ is equivalent to the condition that
\[
    \Gamma(\cX,\cO_\cX)/f\cdot\Gamma(\cX,\cO_\cX)
\]
be finite as a $\Gamma(U,\cO_U)$-module.
\end{remark}

\begin{proposition}\label{pr:nicetriples}
Assumptions being as in Theorem~\ref{th:GrSerre}, let $U:=\spec\cO_{X,x}$ and let~$\cG$ be a principal $\bG$-bundle over $U$ having a rational section. Then there are a nice triple $(q_U\colon \cX\to U,f,\Delta)$ and a $\bG$-bundle $\cE$ over $\cX$ such that\\
\stepzero\noindstep $\Delta^*\cE\simeq\cG$;\\
\noindstep $\cE$ is trivial over the complement of the zero locus $\cZ$ of $f$.

Moreover, if the field $k(x)$ is finite, then we may choose this nice triple so that\\
\noindstep\label{cond3} There is at most one point $z\in\cZ_x$ rational over $k(x)$;\\
\noindstep\label{cond4} For any integer $r\ge 1$ one has
    \[
        \#\{z\in \cZ_x|\;[k(z):k(x)]=r\}\le\#\{z\in\A^1_x|\;[k(z):k(x)]=r\}.
    \]
\end{proposition}
The proof, given below, is similar to~\cite[Thm.~4.3]{PaninNiceTriples},~\cite[Prop~6.1]{PaninStavrovaVavilov}, and~\cite[Sect.~3--4]{PaninMovingLemma}.

\begin{proof}
By Proposition~\ref{pr:QElFib} there are
\begin{itemize}
\item an open affine subscheme $X'\subset X$ containing $x$;
\item a quasi-elementary fibration $p\colon X'\to S$ with $S$ connected and smooth over~$\Lambda$;
\item a principal divisor $Z'\subset X'$ finite over $S$;
\item a $\bG$-bundle $\cF$ over $X'$ extending $\cG$ and such that $\cF$ is trivial over $X'-Z'$;
\item a finite surjective $S$-morphism $X'\to\A^1_S$.
\end{itemize}

Put $\cX':=X'\times_S U$, let $q'_U\colon \cX'\to U$ be the projection. Let $g\in H^0(X',\cO_{X'})$ be an equation of $Z'$, set $f'=p_1^*(g)\in H^0(\cX',\cO_{\cX'})$. Let $\Delta$ be the composition
\[
    U\xrightarrow{\mathrm{diag}}U\times_SU\xrightarrow{\mathrm{can}\times\Id_U}X'\times_S U=\cX'.
\]
Let $\cX$ be the connected component of $\cX'$ containing $\Delta(U)$. Then $\cX$ is irreducible because it is regular and connected. Since $p\colon X'\to S$ is flat (even smooth) of relative dimension one, $q'_U$ is also so, and we see that every component of each fiber is
one-dimensional. Next, $\Delta^*(f')=g|_U\ne0$ because $g\ne0$ and $X'$ is integral. By construction $(q'_U|_\cX,f'|_\cX,\Delta)$ is a~nice triple. Let $\cE'$ be the pullback of $\cF$ to $\cX'$ and $\cE$ be the restriction of $\cE'$ to $\cX$. It is clear that $\cE$ satisfies the conditions of our proposition, so this completes the proof in the case when the field $k(x)$ is infinite.

Consider the case when $k(x)$ is finite. Let $\cT$ be a finite subscheme of $\cX$ intersecting every component of $\cX_x$. Set $\cY:=\Delta(U)\cup\cZ\cup\cT$. Clearly, $\cY$ is finite over $U$; let $\{y_1,\dotsc,y_m\}$ be all of its closed points; let $\cS:=\spec(\cO_{y_1,\dotsc,y_m})$ be the corresponding semilocal scheme. Clearly, $\Delta$ factors through $\cS$.

\begin{lemma}
There exists a finite \'etale morphism $\rho\colon \cS'\to\cS$ and a~section $\Delta'\colon U\to\cS'$ such that $\rho\circ\Delta'=\Delta$, $\Delta'(x)$ is the only $k(x)$-rational point of the fiber $\cS'_x$, and for any integer $r\ge 1$ one has
\begin{equation}\label{eq:cond}
\#\{z\in\cS'_x|\;[k(z):k(x)]=r\}\le\#\{z\in\A^1_x|\;[k(z):k(x)]=r\}.
\end{equation}
\end{lemma}
\begin{proof}
Let $A:=\cO_{y_1,\dotsc,y_m}$ so that $\cS=\spec A$, let $I$ be the ideal of $\Delta(U)$, so that $A=I\oplus R$. Let $\fm_i$ be the maximal ideal of $A$ corresponding to $y_i$ so that $\fm_1$, \ldots, $\fm_m$ are all the maximal ideals of $A$. We may assume that $y_1=\Delta(x)$ so that $\fm_1$ is the ideal of $\Delta(x)$, that is, $\fm_1\supset I$.

Choose a large number $N>0$ and for each $i=2,\dotsc,m$ a monic polynomial $f_i\in(A/\fm_i)[t]$ of degree $N$ and such that
\begin{itemize}
\item if $A/\fm_i$ is finite, then $f_i$ is irreducible;
\item if $A/\fm_i$ is infinite, then $f_i$ is a product of distinct monic polynomials of degree one.
\end{itemize}
Take $f_1\in(A/\fm_1)[t]$ of the form $tg$, where $g$ is irreducible of degree $N-1$. By the Chinese Remainder Theorem applied coefficientwise we can find a monic polynomial $f\in A[t]$ such that $\deg f=N$, $f\in I+tA[t]$ and $f\bmod\fm_i=f_i$ for all $i$. Set $\cS':=\spec(A[t]/(f))$. Clearly, $\cS'$ is finite and flat over $\cS$. Thus, to check that $\cS'$ is \'etale over $\cS$ it is enough to check that the fiber of $\cS'$ over each $y_i\in\cS$ is reduced. But this follows from the definition of $f_i$.

The morphism $\Delta'$ is induced by the composition
\[
    A[t]/(f)\to A[t]/(I+tA[t])=R.
\]
Next, for every $i>1$ such that $A/\fm_i$ is finite, there is only one point of $\cS_x'$ lying over $y_i$. On the other hand, if a point of $\cS_x'$ lies over $y_i$ such that $A/\fm_i$ is infinite, then the degree of this point over $x$ is infinite as well (because we assumed that $k(x)$ is finite). Thus we have
\[
        \#\{z\in\cS'_x|\;[k(z):k(x)]=r\}\begin{cases}
        =1\text{ if }r=1,\\
        =0\text{ if }2\le r\le N-2,\\
        \le m\text{ if } r\ge N-1.
        \end{cases}
\]
It follows that $\Delta'(x)$ is the unique $k(x)$-rational point of the fiber $\cS'_x$ and that condition~\eqref{eq:cond} is satisfied for $N$ large enough.
\end{proof}

Take $\rho$, $\cS'$ and $\Delta'$ as in the above lemma. We can extend~$\rho$ and~$\cS'$ to a neighborhood of $\cS$ to get a diagram
    \[
        \xymatrix{ & \cS'\;\ar@{^{(}->}[r]\ar[d]^\rho & \cV'\ar[d]^\theta &\\
        U\ar[ur]^{\Delta'}\;\ar[r]^\Delta & \cS\;\ar@{^{(}->}[r] &\cV\;\ar@{^{(}->}[r] & \cX,}
    \]
where $\cV$ is an open subscheme of $\cX$, $\theta$ is finite \'etale. Note that $\cS\subset\cV$ implies that $\cY\subset\cV$ by the definition of $\cS$.
\begin{lemma}\label{lm:normalization}
There is an open subscheme $\cW\subset\cV$ such that $\cW\supset\cY$ and $\cW$ admits a finite $U$-morphism to $\A^1_U$.
\end{lemma}
\begin{proof}
By definition of nice triples we have a dominant morphism $\cX\to\A_U^1$, which gives an embedding of the field of functions of $\A_U^1$ into the field of functions of $\cX$. Let~$\overline\cX$ be the normalization of $\P^1_U$ in the field of functions of $\cX$. Note that $U$ is excellent and therefore Nagata ring, so normalization gives a finite morphism $\tilde\Pi\colon \overline\cX\to\P^1_U$. Since $\cX$ is normal, $\tilde\Pi^{-1}(\A_U^1)=\cX$. Thus $\overline\cX-\cX$ is finite over $\infty\times U$ and thus over $U$. Next, $\overline\cX_x-\cV_x=(\overline\cX_x-\cX_x)\cup(\cX_x-\cV_x)$ is finite (the second term is finite because it does not intersect $\cT_x$). It follows that $\overline\cX-\cV$ is finite over $U$ (indeed, it is projective and the closed fiber is finite). Using Lemma~\ref{lm:finitedivisor}\eqref{FD3}, we find a finite morphism $\overline\Pi\colon \overline\cX\to\P^1_U$ such that $\overline\Pi(\cY)\subset0\times U$ and $\overline\Pi(\overline\cX-\cV)\subset\infty\times U$. It remains to take $\cW:=\overline\Pi^{-1}(\A^1_U)$.
\end{proof}
Let $\cW$ be as in the above lemma. Let $\cX''$ be the connected component of $\theta^{-1}(\cW)$ containing $\Delta'(U)$. Set $q''_U:=q_U\circ\theta|_{\cX''}$ and $f''=f\circ\theta|_{\cX''}$. Then $(q''_U\colon \cX''\to U,f'',\Delta')$ is the sought-for nice triple. The proof of Proposition~\ref{pr:nicetriples} is complete.
\end{proof}

Let $(q_U,f,\Delta)$ be a nice triple provided by the above proposition. We may assume that $f$ vanishes at $\Delta(x)$ (so that $\Delta(x)\in\cZ$), otherwise the statement of Theorem~\ref{th:GrSerre} is obvious. If $k(x)$ is finite, then by condition~\eqref{cond3} of Proposition~\ref{pr:nicetriples} $\Delta(x)$ is the only $k(x)$-rational point of $\cZ_x$.

\begin{proposition}\label{pr:nicetripledescend}
Notation being as in Theorem~\ref{th:GrSerre}, set $R:=\cO_{X,x}$ and $U:=\spec R$. Let $(q_U,f,\Delta)$ be a nice triple over $U$ such that $\Delta(x)\in\cZ$. Assume that this nice triple satisfies conditions~\eqref{cond3} and~\eqref{cond4} of Proposition~\ref{pr:nicetriples} if $k(x)$ is finite. Then there are a finite surjective $U$-morphism $\sigma\colon \cX\to\A^1_U$, a monic polynomial $h\in R[t]$ vanishing on $\sigma(\cZ)$, and an element $g\in\Gamma(\cX,\cO_\cX)$ such that\\
\stepzero\noindstep the morphism $\sigma_g:=\sigma|_{\cX_g}$ is \'etale, where $\cX_g$ is the open subscheme of $\cX$ given by $\{g\ne0\}$;\\
\noindstep\label{cond:desc} the data $(R[t],\sigma_g^*\colon R[t]\to\Gamma(\cX,\cO_\cX)_g,h)$ satisfy the hypothesis of~\cite[Prop.~2.6]{ColliotTeleneOjanguren}, that is, $R[t]$ is Noetherian, $\Gamma(\cX,\cO_\cX)_g$ is finitely generated as an $R[t]$-algebra, $\sigma_g^*(h)$ is not a zero divisor in $\Gamma(\cX,\cO_\cX)_g$, and $\sigma_g^*$ induces an isomorphism
    \begin{equation*}
            R[t]/(h)\simeq\Gamma(\cX,\cO_\cX)_g/(\sigma_g^*(h)\cdot\Gamma(\cX,\cO_\cX)_g);
    \end{equation*}
    \noindstep $\Delta(U)\cup\cZ\subset\cX_g$.
\end{proposition}
\begin{proof}
If $R$ contains a field, then this follows from the proofs of~\cite[Thm~3.8 and Cor.~7.2]{PaninMovingLemma}. In our case the proof is completely similar but we will still give it for the sake of completeness. Let, as in the proof of Lemma~\ref{lm:normalization}, $\overline\cX$ be the normalization of $\P^1_U$ in the field of functions of $\cX$, so we have a Cartesian diagram
    \[
    \begin{CD}
    \cX @>\Pi >>\A^1_U\\
    @VVV @VVV\\
    \overline\cX @>\tilde\Pi >>\P^1_U
    \end{CD}
    \]
with finite surjective horizontal morphisms and vertical morphisms being open embeddings.

Consider the reduced finite scheme $(\cZ_x)_{red}$. We can find a closed embedding $\iota_1\colon(\cZ_x)_{red}\to\A_x^1$. Indeed, if $k(x)$ is finite, this follows from condition~\eqref{cond4} in Proposition~\ref{pr:nicetriples} together with the fact that a finite extension of a finite field is determined up to isomorphism by its degree. If $k(x)$ is infinite, the statement follows from the fact that for any finite extension of $k(x)$ there are infinitely many points in $\A^1_x$ whose residue field is isomorphic to this extension.

Next, let $(\cZ_x)_{(2)}$ be the first infinitesimal neighborhood of $(\cZ_x)_{red}$ in $\cX_x$. We can extend $\iota_1$ to a closed embedding
$\iota_2\colon(\cZ_x)_{(2)}\to\A_x^1$ because $\cX_x$ is smooth of dimension one over $k(x)$.

Let $\cO(1)$ be the canonical line bundle on $\P^1_U$ and set $\cL:=\tilde\Pi^*\cO(1)$. Let $s_0$ (resp.~$s_\infty$) be the section of $\cL$ vanishing exactly on $\tilde\Pi^{-1}(\infty\times U)$ (resp. on $\tilde\Pi^{-1}(0\times U))$.

Since $\cX_x$ is of pure dimension one and $\cZ_x$ is a finite scheme, we can find a closed subset $W\subset\cX_x$ such that $W\cap\cZ_x=\emptyset$ and $W$ has exactly one point on each irreducible component of $\cX_x$.

\begin{lemma}\label{lm:section}
    For $n\gg0$ there is a section $s_1\in H^0(\overline\cX,\cL^{\otimes n})$ such that\\
    \stepzero\noindstep The restriction of $s_1$ to $\tilde\Pi^{-1}(\infty\times U)$ coincides with $s_\infty^n$.\\
    \noindstep $s_1$ equals zero on $W$.\\
    \noindstep\label{S3} The restriction of $s_1$ to $\cZ_{(2)}$ is equal to $\iota_2^*(t)\cdot s_0^n$, where $t$ is a coordinate on $\A_x^1$.
    \end{lemma}
\begin{proof}
        Let $\cI$ be the ideal sheaf of $\tilde\Pi^{-1}(\infty\times U)\cup W\cup\cZ_{(2)}$ and let $p_U\colon\overline\cX\to U$ and  $pr_U\colon\P^1_U\to U$ be the projections. Then by the projection formula for $n$ large enough we have
        \[
            R^1(p_U)_*(\cL^{\otimes n}\otimes\cI)=R^1(pr_U)_*(\cO(1)^{\otimes n}\otimes(\tilde\Pi)_*\cI)=0.
        \]
        The rest of the proof is completely similar to the proof of Lemma~\ref{lm:finitedivisor}\eqref{FD1}.
    \end{proof}
Let $s_1$ be as in the lemma, we set $\sigma:=s_1/s_0^n$.

\emph{Claim 1. The morphism $\sigma$ is finite, flat, and surjective.} Indeed, consider the projective morphism $\bar\sigma\colon\overline\cX\to\P^1_U$ given by $[s_1:s_0^n]$. Note that, since $\tilde\Pi$ is finite, every one-dimensional irreducible component of $\overline\cX_x$ contains a point of $\tilde\Pi^{-1}(\infty\times x)$. On the other hand, every such component contains a point of $W$. Now it follows from the construction that $\bar\sigma$ is non-constant on each one-dimensional component of $\overline\cX_x$ (because every such component contains a point of $W$ and a point of $\tilde\Pi^{-1}(\infty\times x)$). Arguing as in the proof of Lemma~\ref{lm:finitedivisor}\eqref{FD3}, we see that $\bar\sigma$ is finite and surjective. Since $\cX=\bar\sigma^{-1}(\A^1_U)$, we see that $\sigma$ is also finite and surjective. Since $\cX$ and $\A^1_U$ are regular schemes, the flatness follows from~~\cite[Thm.~23.1]{MatsumuraCommRingTh}. Claim~1 is proved. \qed

Since $\sigma$ is flat, the set of points, where it is \'etale, is open. Denote this open subset by $\cX'$.

\emph{Claim 2. $\Delta(U)\cup\cZ\subset\cX'$.} First of all, the morphism $\sigma$ is \'etale at the points of~$\cZ_x$. Indeed, since $\sigma$ is flat, it is enough to show that $\sigma_x\colon\cX_x\to\A^1_x$ is \'etale at the points of~$\cZ_x$. This follows easily from condition~\eqref{S3} of Lemma~\ref{lm:section}. Since all the closed points of $\cZ$ are in $\cZ_x$, it follows that $\cZ\subset\cX'$. Since the only closed point $\Delta(x)$ of $\Delta(U)$ is also in $\cZ_x$, we see that $\Delta(U)\subset\cX'$. \qed

\emph{Claim 3. $\sigma|_\cZ$ is a closed embedding.} Recall that $U=\spec R$. Let $\fm_x$ be the maximal ideal of $x\in U$. We first show that $\sigma|_{\cZ_x}\colon\cZ_x\to\A^1_x$ is a closed embedding. Since this morphism is set-theoretically injective, it is enough to show that for every closed point $y\in\cZ_x$ the induced morphism $(R/\fm_x)[t]\to\cO_{\cZ_x,y}$ is surjective. By construction the composition
\[
    (R/\fm_x)[t]\to\cO_{\cZ_x,y}\to\cO_{\cZ_x,y}/\fm_y^2
\]
is surjective, where $\fm_y$ is the maximal ideal of $\cO_{\cZ_x,y}$ and the statement follows from the Nakayama Lemma.

It follows that the morphism $(R/\fm_x)[t]\to\Gamma(\cZ_x,\cO_{\cZ_x})$ induced by $\sigma$ is surjective. By the Nakayama Lemma it implies that
the morphism of $R$-modules $R[t]\to\Gamma(\cZ,\cO_\cZ)$ induced by $\sigma$ is also surjective because $\Gamma(\cZ,\cO_\cZ)$ is a finite $R$-module. Claim~3 follows. \qed

Thus we can identify $\sigma(\cZ)$ with a closed subscheme of $\A^1_U$. Moreover, $\cZ\simeq\sigma(\cZ)$.

\emph{Claim 4. We have $\sigma^{-1}(\sigma(\cZ))=\cZ\sqcup\cZ'$ for some closed subscheme $\cZ'\subset\cX$ and $\cZ'\cap\Delta(U)=\emptyset$.} Indeed, the \'etale morphism $\sigma|_{\cX'}$ has a section over $\sigma(\cZ)$. This section can be viewed as a morphism $s\colon\cZ\to\sigma^{-1}(\sigma(\cZ))$. By~\cite[Cor.~17.3.5.]{EGAIV.4}, this morphism is \'etale, so it is an open morphism. But it is also a closed embedding, so $\sigma^{-1}(\sigma(\cZ))=\cZ\sqcup\cZ'$ for some closed subscheme $\cZ'$. The unique closed point $\Delta(x)$ of $\Delta(U)$ is in $\cZ$, so it is not in $\cZ'$. It follows that $\cZ'\cap\Delta(U)=\emptyset$. \qed

\emph{Claim 5. There is a monic polynomial $h\in R[t]$ such that the zero locus of $h$ coincides with $\sigma(\cZ)$.} Let $\cZ_1$, \ldots, $\cZ_n$ be the irreducible components of $\cZ_{red}$. Since $\cX$ is regular, it is locally factorial, so the principal ideal $(f)$ can be written as $\fp_1^{r_1}\dotso\fp_n^{r_n}$, where $\fp_i\subset\Gamma(\cX,\cO_\cX)$ is the prime ideal corresponding to $\cZ_i$ and $r_i$ are some positive integers. Note that $\fp_i$ is of height one.

Let $\fq_i$ be the preimage of $\fp_i$ under $\sigma^*\colon R[t]\to\Gamma(\cX,\cO_\cX)$. By the going-down property of flat extensions, $\fq_i$ is a height one prime ideal. Since $R[t]$ is factorial, the ideal $\fq_i$ is principal. Write $\fq_i=(h_i)$ and set $h=h_1^{r_1}\dotso h_n^{r_n}$.
By Claim~3 $\sigma|_{\cZ_i}$ is a closed embedding, so $(h_i)$ is the ideal of $\sigma(\cZ_i)$.

Next, the closed embedding $\sigma|_\cZ:\cZ\to\A^1_U$ corresponds to the surjective homomorphism of rings $R[t]\to\Gamma(\cZ,\cO_\cZ)=\Gamma(\cX,\cO_\cX)/(f)$. Clearly, $h$ is in the kernel of this morphism. We need to show that the induced homomorphism
\[
    R[t]/(h)\to\Gamma(\cX,\cO_\cX)/(f)
\]
is an isomorphism. Since $\cX$ is affine, we can find $g'\in\Gamma(\cX,\cO_\cX)$ such that $g'|_\cZ=1$ and $g'|_{\cZ'}=0$, where $\cZ'$ is as in Claim~4. Let $\cX_{g'}$ be the corresponding principal open subset of $\cX$ and let $\sigma_{g'}:=\sigma|_{\cX_{g'}}$. Then the canonical morphism $\Gamma(\cX,\cO_\cX)/(f)\to\Gamma(\cX_{g'},\cO_{\cX_{g'}})/(f)$ is an isomorphism, so it is enough to show that the composed homomorphism
\begin{equation}\label{eq:homo}
    R[t]/(h)\to\Gamma(\cX,\cO_\cX)/(f)\to\Gamma(\cX_{g'},\cO_{\cX_{g'}})/(f)
\end{equation}
is an isomorphism. Consider the filtration of the $R[t]$-module $R[t]/(h)$ by the quotients of principal ideals:
\begin{multline*}
    R[t]/(h)\supset(h_1)/(h)\supset\dotso\supset(h_1^{r_1})/(h)\supset(h_1^{r_1}h_2)/(h)\supset
   \\
   \dotso\supset(h_1^{r_1}h_2^{r_2})/(h)\supset\dotso\supset(h)/(h)=0.
\end{multline*}
We also have a similar filtration of the $R[t]$-module $M:=\Gamma(\cX_{g'},\cO_{\cX_{g'}})/(f)$:
\[
    M\supset h_1M\supset\dotso\supset h_1^{r_1}M\supset h_1^{r_1}h_2M\supset
   \dotso\supset h_1^{r_1}h_2^{r_2}M\supset\dotso\supset hM=0.
\]
The homomorphism~\eqref{eq:homo} is a homomorphism of filtered $R[t]$-modules, so we only need to check that it induces an isomorphism on the associated graded modules. This boils down to checking that for each $i$ the canonical homomorphism
\[
    R[t]/(h_i)\to\Gamma(\cX_g,\cO_{\cX_g})/(h_i)
\]
is an isomorphism. Since $(h_i)$ is the ideal of $\sigma(\cZ_i)$, this is equivalent to the fact that $\sigma$ induces an isomorphism $\sigma_g^{-1}(\sigma(\cZ_i))\to\sigma(\cZ_i)$, which, in turn, follows from the definition of $g$. Claim~5 is proved. \qed

Now we can finish the proof Proposition~\ref{pr:nicetripledescend}. The closed subscheme $\Delta(U)\cup\cZ$ is contained in the open subset $\cX'-\cZ'$. Thus we can find $g\in H^0(\cX,\cO_\cX)$ such that $\Delta(U)\cup\cZ\subset\cX_g\subset\cX'-\cZ'$. By definition of $\cX'$ the morphism $\sigma_g:=\sigma|_{\cX_g}$ is \'etale. Thus, we only need to verify condition~\eqref{cond:desc} of the proposition. Obviously, $R[t]$ is Noetherian and $\Gamma(\cX,\cO_\cX)_g$ is finitely generated as an $R[t]$-algebra. Since $\cX$ is integral and $\sigma$ is surjective, $\sigma_g^*(h)$ is not a zero divisor in $\Gamma(\cX,\cO_\cX)_g$. We have
\[
    \sigma_g^{-1}(\cZ)=(\cZ\sqcup\cZ')\cap\cX_g=\cZ.
\]
Thus $\sigma_g$ induces an isomorphism $\sigma_g^{-1}(\sigma(\cZ))\to\sigma(\cZ)$. This is equivalent to the isomorphism of part~\eqref{cond:desc} of Proposition~\ref{pr:nicetripledescend}. The proof is complete.
\end{proof}

\subsection{End of the proof of Theorem~\ref{th:GrSerre}}
\begin{proposition}
The notation and assumptions being as in Theorem~\ref{th:GrSerre}, put $U:=\spec\cO_{X,x}$ and let $\cG$ be a generically trivial principal $\bG$-bundle over $U$. Then there is a $\bG$-bundle $\cF$ over $\A^1_U$ such that
\begin{itemize}
\item $\cF$ is trivial over the complement of a closed subscheme $\cY\subset\A^1_U$ such that $\cY$ is finite over $U$;
\item $\cF|_{0\times U}\simeq\cG$.
\end{itemize}
\end{proposition}
\begin{proof}
By Proposition~\ref{pr:nicetriples}, there is a nice triple $(q_U\colon \cX\to U,f,\Delta)$ and a $\bG$-bundle $\cE$ over $\cX$ such that\\
\stepzero\noindstep $\Delta^*\cE\simeq\cG$;\\
\noindstep $\cE$ is trivial over the complement of the zero locus $\cZ$ of $f$.

Moreover, if the field $k(x)$ is finite, this nice triple satisfies assumptions~\eqref{cond3} and~\eqref{cond4} of Proposition~\ref{pr:nicetriples}. As we explained before the Proposition~\ref{pr:nicetripledescend}, we may assume that $\Delta(x)\in\cZ$. Let a
$U$-morphism $\sigma\colon \cX\to\A^1_U$, a monic polynomial $h\in R[t]$ vanishing on $\sigma(\cZ)$, and an element $g\in\Gamma(\cX,\cO_\cX)$
be those provided by Proposition~\ref{pr:nicetripledescend}. After performing an affine transformation of $\A^1_U$, we may assume that $\Delta^*(\sigma)$ coincides with the closed embedding $0\times U\hookrightarrow\A^1_U$. Condition~\eqref{cond:desc} of Proposition~\ref{pr:nicetripledescend} together with~\cite[Prop.~2.6]{ColliotTeleneOjanguren} shows that the diagram
    \[
    \begin{CD}
    \cX_{g\sigma^*(h)}@>>>\cX_g\\
    @VVV @V\sigma_gVV\\
    (\A^1_U)_h @>>>\A^1_U
    \end{CD}
    \]
can be used to glue principal $\bG$-bundles in the following sense: given a $\bG$-bundle over $(\A^1_U)_h$, a $\bG$-bundle over $\cX_g$, and an isomorphism of their pullbacks to $\cX_{g\sigma^*(h)}$, we can glue the bundles to make a $\bG$-bundle over $\A^1_U$. In particular, since $\cX_{g\sigma^*(h)}\subset\cX_f$, we can glue $\cE|_{\cX_g}$ with the trivial $\bG$-bundle over $(\A^1_U)_h$ to make a desired $\bG$-bundle~$\cF$ over $\A^1_U$.

Clearly, all the conditions of the proposition are satisfied with $\cY:=\{h=0\}$, which is finite over $U$ because $h$ is monic.
\end{proof}

It remains to apply Theorem~\ref{th:A1} to $R=\cO_{X,x}$, $\bH:=\bG_{X,x}=\bG\times_\Lambda U$, and $\cF$. The proof of Theorem~\ref{th:GrSerre} is complete. \qed

\begin{remark}\label{rm:nicetriples}
When the residue field $k(x)$ is infinite, one can prove the main theorem without using the nice triples by descending the $\bG$-bundle $\cF$ from Proposition~\ref{pr:QElFib} to $\A^1_S$ directly and applying Theorem~\ref{th:A1}. This would not work if $k(x)$ is finite because the analogue of conditions~\eqref{cond3} and~\eqref{cond4} of Proposition~\ref{pr:nicetriples} might fail for the special fiber of the quasi-elementary fibration $p:X'\to S$. The advantage of nice triples to quasi-elementary fibrations, is that the original principal bundle is the pullback via a closed morphism $\Delta$. Thus we were able to ``improve'' the original nice triple by replacing it with an \'etale base change having a section over $\Delta(U)$.
\end{remark}

\begin{remark}\label{rm:NoNisnevich}
  It is not necessary to use~\cite{Nisnevich1} in the case, when $\dim X=1$, as we can easily re-prove the required statement in this case. Indeed, let a $\bG_{X,x}$-bundle $\cG$ have a rational section. Then it is generically trivial, so it admits a generic reduction to any Borel subgroup of $\bG_{X,x}$. By Proposition~\ref{pr:codim2}, such a reduction extends to $\spec\cO_{X,x}$ (because $\dim\cO_{X,x}=1$). By~\cite[Exp.~XXVI, Cor.~2.3]{SGA3-3} we see that $\cG$ has a reduction to a split maximal torus of $\bG_{X,x}$. Now it is easy to see that $\cG$ is trivial.
\end{remark}

\section{Proofs of Theorems~\ref{th:QuadForms} and~\ref{th:QuadForms2}}\label{sect:QuadForms} We keep the notation and assumptions from the statements of the theorems.

\begin{proof}[Proof of Theorem~\ref{th:QuadForms2}]
Let $\bO_n$ be the $R$-group scheme of orthogonal transformations of $Q_n$. The scheme of isomorphisms $\Isom(Q,Q_n)$ is a principal $\bO_n$-bundle over $\spec R$. This bundle is locally trivial in the fppf topology. (Note that if $n$ is odd and $2\notin R^\times$, then~$\bO_n$ is not smooth over $R$.) Thus, we only need to show that the natural morphism $H^1_{\text{fppf}}(R,\bO_n)\to H^1_{\text{fppf}}(K,\bO_n)$ has a trivial kernel.

Note that $\bSO_n$ is a split reductive group scheme. If $n$ is odd, we have $\bO_n\simeq\mu_2\times\bSO_n$, where $\mu_2$ is the group scheme of square roots of unity. Since we assume that the Grothendieck--Serre conjecture holds for $R$ and $\bSO_n$, the natural morphism $H^1_{\text{fppf}}(R,\bSO_n)\to H^1_{\text{fppf}}(K,\bSO_n)$ has a trivial kernel (recall that for smooth group schemes there is no difference between fppf principal bundles and \'etale principal bundles). On the other hand, we have
$H^1_{\text{fppf}}(R,\mu_2)=R^\times/(R^\times)^2$ (because $H^1(R,\cO_R^\times)=1$, since $R$ is local). Similarly, $H^1_{\text{fppf}}(K,\mu_2)=K^\times/(K^\times)^2$. It follows now from factoriality of $R$ that the morphism
$H^1_{\text{fppf}}(R,\mu_2)\to H^1_{\text{fppf}}(K,\mu_2)$ has a trivial kernel. This completes the proof in the case, when $n$ is odd.

If $n$ is even, we have an exact sequence $1\to\bSO_n\to\bO_n\to\Z/2\Z\to1$ by~\cite[Ch.~4, Prop.~5.2.2]{KnusBook}. This gives an exact sequence of cohomology
\[
\xymatrix{\Z/2\Z(R) \ar[r]\ar[d]^= & H^1_{\text{fppf}}(R,\bSO_n) \ar[r]\ar[d] & H^1_{\text{fppf}}(R,\bO_n) \ar[r]\ar[d] & H^1_{\text{fppf}}(R,\Z/2\Z) \ar[d]\\
\Z/2\Z(K)\ar[r] & H^1_{\text{fppf}}(K,\bSO_n) \ar[r] & H^1_{\text{fppf}}(K,\bO_n) \ar[r] & H^1_{\text{fppf}}(K,\Z/2\Z).
}
\]
Note that the right vertical arrow has a trivial kernel. Next, the morphism $\bO_n(K)\to\Z/2\Z(K)$ is surjective (again by~\cite[Ch.~4, Prop.~5.2.2]{KnusBook}). Again, by our assumption the middle vertical arrow has a trivial kernel. Now an easy diagram chase proves the claim. The proof of Theorem~\ref{th:QuadForms2} is complete.
\end{proof}

Assume now that 2 is invertible in $R$.
\begin{proof}[Proof of Theorem~\ref{th:QuadForms}]
According to~\cite[I, Cor.~4.7(i)]{baeza2006quadratic}, the orthogonal sum $Q\bot(-Q)$ is isomorphic to $Q_{2n}$. Applying Theorem~\ref{th:QuadForms2}, we see that $Q'\bot(-Q)\simeq Q\bot(-Q)$. Since 2 is invertible in $R$, we may apply Witt's cancellation theorem (see~\cite[I, Cor.~4.3]{baeza2006quadratic}) to conclude that $Q$ and $Q'$ are isomorphic.
\end{proof}

\bibliographystyle{abbrv}
\bibliography{../../../../RF}

\begin{thebibliography}{10}

\bibitem{SGA4-3}
{\em Th\'eorie des topos et cohomologie \'etale des sch\'emas. {T}ome 3}.
\newblock Lecture Notes in Mathematics, Vol. 305. Springer-Verlag, Berlin,
  1973.
\newblock S{\'e}minaire de G{\'e}om{\'e}trie Alg{\'e}brique du Bois-Marie
  1963--1964 (SGA 4), Dirig{\'e} par M. Artin, A. Grothendieck et J. L.
  Verdier. Avec la collaboration de P. Deligne et B. Saint-Donat.

\bibitem{SGA1}
{\em Rev\^{e}tements \'{e}tales et groupe fondamental ({SGA} 1)}, volume~3 of
  {\em Documents Math\'{e}matiques (Paris) [Mathematical Documents (Paris)]}.
\newblock Soci\'{e}t\'{e} Math\'{e}matique de France, Paris, 2003.
\newblock S\'{e}minaire de g\'{e}om\'{e}trie alg\'{e}brique du Bois Marie
  1960--61. [Algebraic Geometry Seminar of Bois Marie 1960-61], Directed by A.
  Grothendieck, With two papers by M. Raynaud, Updated and annotated reprint of
  the 1971 original [Lecture Notes in Math., 224, Springer, Berlin; MR0354651
  (50 \#7129)].

\bibitem{baeza2006quadratic}
R.~Baeza.
\newblock {\em Quadratic forms over semilocal rings}, volume 655.
\newblock Springer, 2006.

\bibitem{CesnaviciusGrSerre}
K.~{\v{C}esnavi\v{c}ius}.
\newblock {Grothendieck-Serre in the quasi-split unramified case}.
\newblock {\em arXiv e-prints, 2009.05299}, 2020.

\bibitem{ColliotTeleneOjanguren}
J.-L. Colliot-Th{\'e}l{\`e}ne and M.~Ojanguren.
\newblock Espaces principaux homog\`enes localement triviaux.
\newblock {\em Inst. Hautes \'Etudes Sci. Publ. Math.}, (75):97--122, 1992.

\bibitem{ColliotTS1979fibres}
J.-L. Colliot-Th{\'e}l{\`e}ne and J.-J. Sansuc.
\newblock Fibr{\'e}s quadratiques et composantes connexes r{\'e}elles.
\newblock {\em Mathematische Annalen}, 244(2):105--134, 1979.

\bibitem{ColliotTheleneSansuc}
J.-L. Colliot-Th{\'e}l{\`e}ne and J.-J. Sansuc.
\newblock Principal homogeneous spaces under flasque tori: applications.
\newblock {\em J. Algebra}, 106(1):148--205, 1987.

\bibitem{SGA3-3}
M.~Demazure and A.~Grothendieck.
\newblock {\em Sch\'emas en groupes. {III}: {S}tructure des sch\'emas en
  groupes r\'eductifs}.
\newblock S\'eminaire de G\'eom\'etrie Alg\'ebrique du Bois Marie 1962/64 (SGA
  3). Dirig\'e par M. Demazure et A. Grothendieck. Lecture Notes in
  Mathematics, Vol. 153. Springer-Verlag, Berlin, 1970.

\bibitem{FedorovExotic}
R.~Fedorov.
\newblock Affine {G}rassmannians of group schemes and exotic principal bundles
  over {$\mathbb A^1$}.
\newblock {\em Amer.~Journal of Math.}, 138(4):879--906, 2016.

\bibitem{FedorovPanin}
R.~Fedorov and I.~Panin.
\newblock A proof of the {G}rothendieck--{S}erre conjecture on principal
  bundles over regular local rings containing infinite fields.
\newblock {\em Publications math\'ematiques de l'IH\'ES}, 122(1):169--193,
  2015.

\bibitem{GilleTorseurs}
P.~Gille.
\newblock Torseurs sur la droite affine.
\newblock {\em Transform. Groups}, 7(3):231--245, 2002.

\bibitem{GrothendieckTorsion}
A.~Grothendieck.
\newblock Torsion homologique et sections rationnelles.
\newblock In {\em Anneaux de Chow et applications, S\'eminaire Claude
  Chevalley}, number~3. Paris, 1958.

\bibitem{EGAII}
A.~Grothendieck.
\newblock \'{E}l\'ements de g\'eom\'etrie alg\'ebrique. {II}. \'{E}tude globale
  \'el\'ementaire de quelques classes de morphismes.
\newblock {\em Inst. Hautes \'Etudes Sci. Publ. Math.}, (8):222, 1961.

\bibitem{EGAIII-1}
A.~Grothendieck.
\newblock \'{E}l\'ements de g\'eom\'etrie alg\'ebrique. {III}. \'{E}tude
  cohomologique des faisceaux coh\'erents. {I}.
\newblock {\em Inst. Hautes \'Etudes Sci. Publ. Math.}, (11):167, 1961.

\bibitem{EGAIV-1}
A.~Grothendieck.
\newblock \'{E}l\'ements de g\'eom\'etrie alg\'ebrique. {IV}. \'{E}tude local
  des sch\'emas et des morphismes des sch\'emas, premi{\`e}re partie
  (r{\'e}dig{\'e}s avec la collaboration de {J}ean {D}ieudonn{\'e}).
\newblock {\em Institut des Hautes {\'E}tudes Scientifiques. Publications
  Math{\'e}matiques}, 20, 1964.

\bibitem{EGAIV-2}
A.~Grothendieck.
\newblock \'{E}l\'ements de g\'eom\'etrie alg\'ebrique. {IV}. \'{E}tude local
  des sch\'emas et des morphismes des sch\'emas, seconde partie.
\newblock {\em Inst. Hautes \'Etudes Sci. Publ. Math.}, (24):5--231, 1965.

\bibitem{EGAIV-3}
A.~Grothendieck.
\newblock \'{E}l\'ements de g\'eom\'etrie alg\'ebrique. {IV}. \'{E}tude local
  des sch\'emas et des morphismes des sch\'emas, troisi\`eme partie.
\newblock {\em Inst. Hautes \'Etudes Sci. Publ. Math.}, (28):255, 1966.

\bibitem{EGAIV.4}
A.~Grothendieck.
\newblock \'{E}l\'ements de g\'eom\'etrie alg\'ebrique. {IV}. \'{E}tude locale
  des sch\'emas et des morphismes de sch\'emas {IV}.
\newblock {\em Inst. Hautes \'Etudes Sci. Publ. Math.}, (32):361, 1967.

\bibitem{GrothendieckBrauer2}
A.~Grothendieck.
\newblock Le groupe de {B}rauer. {II}. {T}h\'eorie cohomologique.
\newblock In {\em Dix {E}xpos\'es sur la {C}ohomologie des {S}ch\'emas}, pages
  67--87. North-Holland, Amsterdam, 1968.

\bibitem{FGA1}
A.~Grothendieck.
\newblock Technique de descente et th\'eor\`emes d'existence en g\'eometrie
  alg\'ebrique. {I}. {G}\'en\'eralit\'es. {D}escente par morphismes
  fid\`element plats.
\newblock In {\em S\'eminaire {B}ourbaki, {V}ol.\ 5, {E}xp.\ {N}o.\ 190.},
  pages 299--327. Soc. Math. France, Paris, 1995.

\bibitem{Guo2019GrSerreDedekind}
N.~Guo.
\newblock The {G}rothendieck--{S}erre conjecture over semilocal {D}edekind
  rings.
\newblock {\em Transformation Groups}, pages 1--21, 2020.

\bibitem{KnusBook}
M.-A. Knus.
\newblock {\em Quadratic and {H}ermitian forms over rings}, volume 294 of {\em
  Grundlehren der Mathematischen Wissenschaften [Fundamental Principles of
  Mathematical Sciences]}.
\newblock Springer-Verlag, Berlin, 1991.
\newblock With a foreword by I. Bertuccioni.

\bibitem{Liu2002AG}
Q.~Liu and R.~Erne.
\newblock {\em Algebraic geometry and arithmetic curves}, volume~6.
\newblock Oxford University Press on Demand, 2002.

\bibitem{MatsumuraCommRingTh}
H.~Matsumura.
\newblock {\em Commutative ring theory}, volume~8 of {\em Cambridge Studies in
  Advanced Mathematics}.
\newblock Cambridge University Press, Cambridge, second edition, 1989.
\newblock Translated from the Japanese by M. Reid.

\bibitem{MumfordEtAl1994GIT}
D.~Mumford, J.~Fogarty, and F.~Kirwan.
\newblock {\em Geometric invariant theory}, volume~34.
\newblock Springer Science \& Business Media, 1994.

\bibitem{Nisnevich1}
Y.~Nisnevich.
\newblock Espaces homog\`enes principaux rationnellement triviaux et
  arithm\'etique des sch\'emas en groupes r\'eductifs sur les anneaux de
  {D}edekind.
\newblock {\em C. R. Acad. Sci. Paris S\'er. I Math.}, 299(1):5--8, 1984.

\bibitem{Nisnevich2}
Y.~Nisnevich.
\newblock Rationally trivial principal homogeneous spaces, purity and
  arithmetic of reductive group schemes over extensions of two-dimensional
  regular local rings.
\newblock {\em C. R. Acad. Sci. Paris S\'er. I Math.}, 309(10):651--655, 1989.

\bibitem{PaninNiceTriples}
I.~Panin.
\newblock {Nice triples and the Grothendieck--Serre conjecture concerning
  principal G-bundles over reductive group schemes}.
\newblock {\em Duke Math. Journal}, 168(2):351--375, 2019.

\bibitem{PaninStavrovaVavilov}
I.~Panin, A.~Stavrova, and N.~Vavilov.
\newblock On {G}rothendieck-{S}erre's conjecture concerning principal
  {$G$}-bundles over reductive group schemes: {I}.
\newblock {\em Compos. Math.}, 151(3):535--567, 2015.

\bibitem{PaninMovingLemma}
I.~A. Panin.
\newblock Nice triples and moving lemmas for motivic spaces.
\newblock {\em Izvestiya: Mathematics}, 83(4):796, 2019.

\bibitem{PaninFiniteFieldsIzvestiya}
I.~A. Panin.
\newblock Proof of the {G}rothendieck--{S}erre conjecture on principal bundles
  over regular local rings containing a field.
\newblock {\em Izvestiya: Mathematics}, 84(4):780, 2020.

\bibitem{PaninPurity17}
I.~A. Panin.
\newblock Two purity theorems and the {G}rothendieck--{S}erre conjecture
  concerning principal-bundles.
\newblock {\em Sbornik: Mathematics}, 211(12):1777--1794, 2020.

\bibitem{PaninStavrovaDedekind}
I.~A. Panin and A.~K. Stavrova.
\newblock On the {G}rothendieck--{S}erre conjecture concerning principal
  {G}-bundles over semilocal {D}edekind domains.
\newblock {\em Journal of Mathematical Sciences}, 222(4):453--462, 2017.

\bibitem{PoonenOnBertini}
B.~Poonen.
\newblock Bertini theorems over finite fields.
\newblock {\em Ann. of Math. (2)}, 160(3):1099--1127, 2004.

\bibitem{Popescu}
D.~Popescu.
\newblock General {N}\'eron desingularization and approximation.
\newblock {\em Nagoya Math. J.}, 104:85--115, 1986.

\bibitem{SerreFibres}
J.-P. Serre.
\newblock Espaces fibr\'es alg\'ebrique.
\newblock In {\em Anneaux de Chow et applications, S\'eminaire Claude
  Chevalley}, number~3. Paris, 1958.

\bibitem{SpivakovskyPopescu}
M.~Spivakovsky.
\newblock A new proof of {D}. {P}opescu's theorem on smoothing of ring
  homomorphisms.
\newblock {\em J. Amer. Math. Soc.}, 12(2):381--444, 1999.

\bibitem{SwanOnPopescu}
R.~G. Swan.
\newblock N\'eron-{P}opescu desingularization.
\newblock In {\em Algebra and geometry ({T}aipei, 1995)}, volume~2 of {\em
  Lect. Algebra Geom.}, pages 135--192. Int. Press, Cambridge, MA, 1998.

\end{thebibliography}
\end{document}